\documentclass[12pt,reqno]{amsart}
\usepackage{amsmath, amsthm, amssymb}

\advance \topmargin by -\headheight
\advance \topmargin by -\headsep
     
\evensidemargin \oddsidemargin
\newtheorem{theorem}{Theorem}[section]
\newtheorem{lemma}[theorem]{Lemma}

\theoremstyle{definition}

\theoremstyle{remark}

\numberwithin{equation}{section}

\newcommand{\mmod}[1]{\,\,(\text{mod}\,\,#1)}

\def\bfa{{\mathbf a}}
\def\bfb{{\mathbf b}}
\def\bfc{{\mathbf c}}
\def\bfd{{\mathbf d}}  

\def\bfh{{\mathbf h}}

\def\bfr{{\mathbf r}}

\def\bfx{{\mathbf x}}
\def\bfy{{\mathbf y}}
\def\bfz{{\mathbf z}}

\def\calA{{\mathcal A}}  

\def\calB{{\mathcal B}} 
\def\calC{{\mathcal C}}

\def\calF{{\mathcal F}}

\def\calI{{\mathcal I}}
\def\calJ{{\mathcal J}}

\def\calS{{\mathcal S}}

\def\gtil{\tilde{g}}

\def\ftil{\widetilde{f}}
\def\gtil{\widetilde{g}}

\def\dbC{{\mathbb C}}\def\dbN{{\mathbb N}}
\def\dbR{{\mathbb R}}
\def\dbZ{{\mathbb Z}}\def\dbQ{{\mathbb Q}}

\def\grf{{\mathfrak f}}\def\grF{{\mathfrak F}}
\def\grg{{\mathfrak g}}
\def\grh{{\mathfrak h}}

\def\grJ{{\mathfrak J}}

\def\grm{{\mathfrak m}}\def\grM{{\mathfrak M}}
\def\grS{{\mathfrak S}}\def\grP{{\mathfrak P}}
\def\grq{{\mathfrak q}}\def\grQ{{\mathfrak Q}}

\def\grp{{\mathfrak p}}

\def\alp{{\alpha}} \def\bfalp{{\boldsymbol \alpha}}
\def\bet{{\beta}}  \def\bfbet{{\boldsymbol \beta}}
\def\gam{{\gamma}} \def\bfgam{{\boldsymbol \gam}} \def\Gam{{\Gamma}}
\def\del{{\delta}}  
\def\zet{{\zeta}} \def\bfzet{{\boldsymbol \zeta}} 
\def\bfeta{{\boldsymbol \eta}} 
\def\tet{{\theta}} \def\bftet{{\boldsymbol \theta}} \def\Tet{{\Theta}}
\def\kap{{\kappa}}
\def\lam{{\lambda}}  

\def\bfxi{{\boldsymbol \xi}}

\def\Ups{{\Upsilon}} 

\def\ome{{\omega}} 
\def\d{{\partial}}
\def\eps{\varepsilon}

\def\le{\leqslant} \def\ge{\geqslant}

\def\d{{\,{\rm d}}}

\begin{document}
\title[Pairs of diagonal equations]{Rational solutions of pairs of diagonal equations, one cubic and one 
quadratic}
\author[Trevor D. Wooley]{Trevor D. Wooley}
\address{School of Mathematics, University of Bristol, University Walk, Clifton, Bristol BS8 1TW, United 
Kingdom}
\email{matdw@bristol.ac.uk}
\subjclass[2010]{11D72, 11P55, 11L07, 11L15}
\keywords{Weyl sums, Hardy-Littlewood method, Diophantine equations}
\date{}
\begin{abstract} We obtain an essentially optimal estimate for the moment of order $32/3$ of the 
exponential sum having argument $\alp x^3+\bet x^2$. Subject to modest local solubility hypotheses, 
we thereby establish that pairs of diagonal Diophantine equations, one cubic and one quadratic, possess 
non-trivial integral solutions whenever the number of variables exceeds $10$.
\end{abstract}
\maketitle

\section{Introduction} Investigations concerning the integral solubility of simultaneous equations by means 
of the circle method are in general limited to variable regimes beyond the convexity barrier, so that the 
underlying number of variables must exceed twice the total degree of the system. This technical limitation 
has been attained or surmounted in very few cases, almost all quadratic in nature. Thus, one has 
conclusions for a single quadratic form in $3$ or $4$ variables (see \cite{Est1962, HB1996, Klo1927}), 
and for systems of $r$ diagonal quadratic forms in $4r+1$ variables (see 
\cite{BC1992, Coo1971, Coo1973}). Most recently, work of the author joint with Br\"udern 
\cite{BW2007, BW2014a, BW2014b} analyses systems of $r$ diagonal cubic equations in $6r+1$ variables 
in general position, and pairs of such equations in $11$ or $12$ variables possessing a block structure. In 
this memoir we attain this technical limit in the previously inaccessible case of two diagonal equations, one 
cubic and one quadratic, in $11$ variables. It transpires that progress is possible here owing to the 
author's recent proof \cite{Woo2014} of the cubic case of the main conjecture in Vinogradov's mean value 
theorem, though wielding the latter to achieve our present purpose entails further innovations beyond the 
conventional repertoire of Hardy-Littlewood artisans.\par

By relabelling variables as necessary, there is no loss of generality in supposing that the simultaneous 
equations central to this paper take the form
\begin{equation}\label{1.1}
\Tet(\bfx,\bfy)=\Phi(\bfx,\bfz)=0,
\end{equation}
where we write
$$\left.\begin{aligned}
\Tet(\bfx,\bfy)=&a_1x_1^3+\ldots +a_lx_l^3+c_1y_1^3+\ldots +c_my_m^3&\\
\Phi(\bfx,\bfz)=&b_1x_1^2+\ldots +b_lx_l^2&+d_1z_1^2+\ldots +d_nz_n^2
\end{aligned}\right\},$$
with the coefficients $a_i,b_i,c_j,d_k$ all non-zero integers. We write $s=l+m+n$ for the total number of 
variables in the system, and we consider the solubility, in rational integers $x_i,y_j,z_k$, of the 
simultaneous diagonal equations (\ref{1.1}).

\begin{theorem}\label{theorem1.1}
The simultaneous equations $\Tet(\bfx,\bfy)=\Phi(\bfx,\bfz)=0$ possess a non-zero integral solution  
provided only that the following conditions hold:
\begin{enumerate}
\item[(a)] the system $\Tet(\bfx,\bfy)=\Phi(\bfx,\bfz)=0$ has a non-trivial real solution, and
\item[(b)] the polynomial $\Phi(\bfx,\bfz)$ is indefinite, and
\item[(c)]  one has $l+m\ge 7$ and $l+n\ge 5$, and
\item[(d)] one has $s=l+m+n\ge 11$.
\end{enumerate}
\end{theorem}

The discussion of \cite[\S5]{Woo1991b} provides examples that demonstrate conditions (a), (b) and (c) of 
Theorem \ref{theorem1.1} to be necessary. Thus, by virtue of condition (d), we see that this theorem 
establishes a conclusion tantamount to the Hasse Principle for pairs of diagonal equations, one cubic and 
one quadratic, whenever the system has $11$ or more variables. We note in this context that (\ref{1.1}) 
possesses non-trivial solutions in every $p$-adic field $\dbQ_p$ provided only that $s\ge 11$ (see the 
main theorem of \cite{Woo1991a}). When $p$ is a prime number with $p\equiv 1\mmod{3}$, moreover, 
there are examples of the shape (\ref{1.1}) with $s=10$ that possess only the trivial $p$-adic solution 
$(\bfx,\bfy,\bfz)={\bf0}$ (see \cite[Lemma 7.2]{Woo1991a}).

It seems that the system (\ref{1.1}) was first discussed in \cite[Theorem 1]{Woo1991b}, where a 
conclusion analogous to that of Theorem \ref{theorem1.1} was obtained with the condition $s\ge 11$ 
replaced by the more stringent constraint $s\ge 14$. In subsequent work \cite[Theorem 1]{Woo1998}, the 
latter was replaced by the condition $s\ge 13$. Not only is our new conclusion superior to these earlier 
results, but it also attains the technical limit imposed by the convexity barrier for the problem of analysing 
the integral solubility of the system (\ref{1.1}).\par

In most circumstances, one can say much more concerning the density of solutions of the system 
(\ref{1.1}) than is apparent from Theorem \ref{theorem1.1}. When $B$ is a large positive number, let 
$N(B)$ denote the number of integral solutions of the system (\ref{1.1}) with $|\bfx|\le B$.

\begin{theorem}\label{theorem1.2}
Suppose that the system $\Tet(\bfx,\bfy)=\Phi(\bfx,\bfz)=0$ has a non-singular solution in the real field 
$\dbR$, and also in each $p$-adic field $\dbQ_p$. Then provided that $s\ge 11$, $m\le 5$ and $n\le 3$, 
one has $N(B)\gg B^{s-5}$.
\end{theorem}

A formal application of the circle method suggests the conjectural asymptotic formula 
$N(B)\sim CB^{s-5}$, in which $C$ is given by a product of real and $p$-adic densities. Thus, the 
conclusion of Theorem \ref{theorem1.2} shows that $N(B)$ grows asymptotically at the expected rate. We 
remark that our methods would permit the proof of this conjectural formula for $N(B)$ whenever 
$s\ge 11$, and $m=n=0$ or $n\in \{1,2\}$ (see (\ref{7.4}) below for a slightly more restrictive 
asymptotic formula). Hitherto, such an asymptotic formula was available only for $s\ge 15$, though this 
conclusion was apparently absent from the literature.\par

The proof of Theorems \ref{theorem1.1} and \ref{theorem1.2} depend on new mean value estimates for 
Weyl sums, only now available as a consequence of the proof \cite{Woo2014} of the cubic case of the main 
conjecture in Vinogradov's mean value theorem. In order to describe these estimates, we define 
$\grF(\alp,\bet)=\grF(\alp,\bet;X)$ by putting
\begin{equation}\label{1.2}
\grF(\alp,\bet;X)=\sum_{1\le x\le X}e(\alp x^3+\bet x^2),
\end{equation}
where $e(z)$ denotes $e^{2\pi iz}$. Then, when $s$ is a positive real number, we put
\begin{equation}\label{1.3}
T_s(X)=\int_0^1\int_0^1 |\grF(\alp,\bet)|^{2s}\d\alp \d\bet .
\end{equation}

\begin{theorem}\label{theorem1.3}
When $1\le s\le 4$, one has $T_s(X)\ll X^{s+\eps}$. In addition,
\begin{equation}\label{1.4}
T_5(X)\ll X^{31/6+\eps}\quad \text{and}\quad T_s(X)\ll X^{2s-5}\quad (s>16/3).
\end{equation}
\end{theorem}

The bound $T_4(X)\ll X^{4+\eps}$ establishes strongly diagonal behaviour for the exponential sum 
$\grF(\alp,\bet)$ for the first time for moments exceeding the sixth. We direct the reader to 
\cite[Theorem 4.1]{Woo1991b} for the bound $T_3(X)\ll X^{3+\eps}$, and \cite[Theorem 1]{Woo1996} 
for the sharper conclusion $T_3(X)=6X^3+O(X^{7/3+\eps})$. Meanwhile, the estimate 
$T_s(X)\ll X^{2s-5}$ follows for $s>7$ by applying an argument based on that underlying the proof of 
Hua's lemma, in combination with \cite[Theorem 4.1]{Woo1991b}. A classical approach to bounding 
$T_5(X)$, using H\"older's inequality to interpolate between the estimates $T_3(X)\ll X^{3+\eps}$ and 
$T_7(X)\ll X^{9+\eps}$ just cited, yields $T_5(X)\ll X^{6+\eps}$ in place of the first estimate 
of (\ref{1.4}).\par

We remark that Theorem \ref{theorem1.3} also improves on the sharpest estimates previously available 
for moments incorporating smooth Weyl sums. Denote the set of $R$-smooth integers not exceeding $X$ 
by
\begin{equation}\label{1.5}
\calA(X,R)=\{ n\in [1,X]\cap \dbZ:\text{$p$ prime and $p|n\Rightarrow p\le R$}\},
\end{equation}
and put
$$\grf(\alp,\bet)=\sum_{x\in \calA(X,R)}e(\alp x^3+\bet x^2).$$
Then \cite[Theorem 2]{Woo1998} shows that whenever $\eta>0$ is sufficiently small in terms of $\eps$, 
and $R\le X^\eta$, then
$$\int_0^1\int_0^1|\grF(\alp,\bet)^4\grf(\alp,\bet)^6|\d\alp \d\bet \ll X^{17/3+\eps}.$$
As is clear, however, the exponent $17/3$ here may be replaced by $31/6$, by virtue of the estimate 
$T_5(X)\ll X^{31/6+\eps}$ made available in Theorem \ref{theorem1.3}.\par

The proof of Theorem \ref{theorem1.3} depends on two auxiliary mean value estimates established in 
\S2, these being exploited in \S3 by means of an argument motivated by the translation invariance of a 
related Vinogradov system. We then turn to the proof of Theorems \ref{theorem1.1} and 
\ref{theorem1.2}, attending to some preliminary simplifications in \S4, and establishing the former 
theorem as a consequence of the latter. The proof of Theorem \ref{theorem1.2}, using the 
Hardy-Littlewood method, is accomplished in \S\S5--13 in three parts according to a classification of 
systems of type (\ref{1.1}) depending on the values of $m$ and $n$. Each such part proceeds in three 
phases, the first discussing such auxiliary estimates as are required in the argument, the second 
addressing the minor arcs of the Hardy-Littlewood dissection, and the third disposing of the major arc 
contribution.\par

Throughout, the letter $s$ will denote a positive integer, and $\eps$ and $\eta$ will denote sufficiently 
small positive numbers. We take $X$ and $P$ to be large positive real numbers depending at most on $s$, 
$\eps$ and $\eta$. The implicit constants in Vinogradov's notation $\ll$ and $\gg$ will depend at most on 
$s$, $\eps$ and $\eta$, unless otherwise indicated. We adopt the following convention concerning the 
numbers $\eps$ and $R$. Whenever $\eps$ or $R$ appear in a statement, we assert that for each 
$\eps>0$, there exists a positive number $\eta=\eta(\eps,s)$ such that the statement holds whenever 
$R\le P^\eta$. Finally, we employ the convention that whenever $G:[0,1)^k\rightarrow \dbC$ is 
integrable, then
$$\oint G(\bfalp)\d\bfalp =\int_{[0,1)^k}G(\bfalp)\d\bfalp .$$
Here and elsewhere, we use vector notation in the natural way.

\section{Auxiliary mean value estimates} In this section we establish estimates for certain auxiliary mean 
values required in our proof of Theorem \ref{theorem1.3}. We begin by introducing some notation with 
which to describe these mean values. Let $Y$ and $H$ be positive numbers, and consider the exponential 
sum $\grg(\bfalp)=\grg(\bfalp;Y,H)$ defined by
\begin{equation}\label{2.1}
\grg(\bfalp;Y,H)=\sum_{0<|h|\le H}\sum_{1\le y\le Y}e(h\alp_1+hy\alp_2+hy^2\alp_3).
\end{equation}
We seek to obtain estimates for mean values of the shape
\begin{equation}\label{2.2}
I_s(Y,H)=\oint |\grg(\bfalp)|^{2s}\d\bfalp .
\end{equation}

\begin{lemma}\label{lemma2.1}
For each $\eps>0$, one has $I_2(Y,H)\ll H^3Y+(HY)^{2+\eps}$.
\end{lemma}

\begin{proof} By orthogonality, the mean value $I_2(Y,H)$ counts the number of integral solutions of the 
simultaneous equations
\begin{equation}\label{2.3}
\sum_{i=1}^4h_iy_i^j=0\quad (0\le j\le 2),
\end{equation}
with $0<|h_i|\le H$ and $1\le y_i\le Y$ $(1\le i\le 4)$. We divide the solutions counted by $I_2(Y,H)$ into 
three classes. Denote by $T_0$ the number of solutions of the system (\ref{2.3}) counted by $I_2(Y,H)$ 
in which $y_1=y_2=y_3=y_4$, by $T_1$ the corresponding number with $h_3y_3^2+h_4y_4^2=0$, 
and by $T_2$ the number with $y_3\ne y_4$ and $h_3y_3^2+h_4y_4^2\ne 0$. Then by symmetry, it 
follows that
\begin{equation}\label{2.4}
I_2(Y,H)\ll T_0+T_1+T_2.
\end{equation}

\par Observe first that, by considering the linear equation in $\bfh$ in (\ref{2.3}), one finds that the 
number of possible choices for $\bfh$ is $O(H^3)$, and consequently
\begin{equation}\label{2.5}
T_0\ll H^3Y.
\end{equation}

\par Given a solution $\bfh,\bfy$ of (\ref{2.3}) counted by $T_1$, meanwhile, it follows from the equation 
with $j=2$ that
$$h_1y_1^2+h_2y_2^2=0=h_3y_3^2+h_4y_4^2.$$
Given fixed choices of $h_2,y_2,h_4,y_4$, it follows that $h_1$ and $y_1$ are divisors of the fixed 
non-zero integer $h_2y_2^2$, and that $h_3$ and $y_3$ are divisors of the fixed non-zero integer 
$h_4y_4^2$. An elementary estimate for the divisor function therefore shows that there are 
$O((HY)^\eps)$ possible choices for $h_1,y_1,h_3,y_3$, and hence
\begin{equation}\label{2.6}
T_1\ll (HY)^{2+\eps}.
\end{equation}

\par In order to estimate $T_2$, we begin by considering the polynomial identity
$$(a+b)(ax^2+by^2)-(ax+by)^2=ab(x-y)^2.$$
Given any solution $\bfy,\bfh$ of the system (\ref{2.3}), one therefore has
\begin{equation}\label{2.7}
h_1h_2(y_1-y_2)^2=h_3h_4(y_3-y_4)^2.
\end{equation}
Consider a fixed choice of $h_3,h_4,y_3,y_4$ corresponding to a solution $\bfh,\bfy$ counted by $T_2$. 
Since $y_3\ne y_4$, it follows from (\ref{2.7}) that $h_1,h_2$ and $z=y_1-y_2$ are each divisors of the 
fixed non-zero integer $h_3h_4(y_3-y_4)^2$. A standard estimate for the divisor function shows that 
there are $O((HY)^\eps)$ possible such choices. Fix any one choice, and consider the equation with $j=2$ 
in (\ref{2.3}). One has
\begin{equation}\label{2.8}
h_1(z+y_2)^2+h_2y_2^2=-h_3y_3^2-h_4y_4^2\ne 0.
\end{equation}
Since the integer $h_1z^2+2h_1zy_2+(h_1+h_2)y_2^2$ is non-zero, it follows that the equation 
(\ref{2.8}) is non-trivial in terms of $y_2$. For if one were to have $h_1+h_2=2h_1z=0$, then one would 
have also $h_1z^2=0$, yielding a contradiction. There are consequently at most $2$ solutions for the 
remaining undetermined variable $y_2$, and then $y_1=z+y_2$ is also determined. We therefore 
conclude that
\begin{equation}\label{2.9}
T_2\ll (HY)^{2+\eps}.
\end{equation}

\par By substituting the estimates (\ref{2.5}), (\ref{2.6}) and (\ref{2.9}) into (\ref{2.4}), we conclude that 
$I_2(Y,H)\ll H^3Y+(HY)^{2+\eps}$, thereby completing the proof of the lemma.
\end{proof}

We employ Lemma \ref{lemma2.1} to establish a bound for $I_3(Y,H)$ via the Hardy-Littlewood method.

\begin{lemma}\label{lemma2.2} For each $\eps>0$, one has $I_3(Y,H)\ll 
(HY)^\eps (H^5Y^2+H^4Y^3)$.
\end{lemma}

\begin{proof} We begin by obtaining an estimate of Weyl-type for the exponential sum $\grg(\bfalp)$. By 
applying Cauchy's inequality to (\ref{2.1}), one obtains the bound
\begin{align*}
|\grg(\bfalp)|^2&\ll H\sum_{0<|h|\le H}
\biggl| \sum_{1\le y\le Y}e(h\alp_1+hy\alp_2+hy^2\alp_3)\biggr|^2\\
&\le 2H^2Y+2H\sum_{0<|h|\le H}\biggl| \sum_{1\le y_2<y_1\le Y}
e(h(y_1-y_2)\alp_2+h(y_1^2-y_2^2)\alp_3)\biggr| .
\end{align*}
Thus, by substituting $z=y_1-y_2$ and $w=y_1+y_2$, we find that
$$|\grg(\bfalp)|^2\ll H^2Y+H\sum_{1\le h\le H}\sum_{1\le z\le Y}
\biggl| \sum_{\substack{w\in I(z)\\ 2|(w+z)}}e(hz\alp_2+hzw\alp_3)\biggr|,$$
where $I(z)$ is an interval of integers lying in $[1,2Y]$. Consequently, one finds that
\begin{align*}
|\grg(\bfalp)|^2&\ll H^2Y+H\sum_{1\le h\le H}\sum_{1\le z\le Y}\min \{Y,\| 2hz\alp_3\|^{-1}\}\\
&\ll H^2Y+(HY)^\eps H\sum_{1\le m\le 2HY}\min \{ HY^2/m,\|m\alp_3\|^{-1}\} .
\end{align*}
An application of \cite[Lemma 2.2]{Vau1997} therefore shows that whenever $a\in \dbZ$ and $q\in \dbN$ 
satisfy $(a,q)=1$ and $|\alp_3 -a/q|\le q^{-2}$, then
\begin{equation}\label{2.10}
|\grg(\bfalp)|^2\ll (HY)^{2+\eps} (q^{-1}+Y^{-1}+q(HY^2)^{-1}).
\end{equation}
Thus, a standard transference principle (see Lemma \ref{lemma16.1} below) reveals that whenever 
$\alp_3\in [0,1)$, $b\in \dbZ$ and $r\in \dbN$ satisfy $(b,r)=1$, then
\begin{equation}\label{2.11}
|\grg(\bfalp)|\ll (HY)^{1+\eps}(\lam^{-1}+Y^{-1}+\lam (HY^2)^{-1})^{1/2},
\end{equation}
where $\lam=r+HY^2|r\alp_3-b|$.\par

We now apply the Hardy-Littlewood method. Denote by $\grM$ the union of the intervals
$$\grM(q,a)=\{ \tet\in[0,1):|q\tet-a|\le (HY)^{-1}\},$$
with $0\le a\le q\le Y$ and $(a,q)=1$. In addition, put $\grm=[0,1)\setminus \grM$. Given 
$\alp_3\in \grm$, an application of Dirichlet's approximation theorem shows that there exist $q\in \dbN$ 
and $a\in \dbN$ with $0\le a\le q\le HY$, $(a,q)=1$ and $|q\alp_3-a|\le (HY)^{-1}$. Since 
$\alp_3\in \grm$, one therefore has $q>Y$, and so it follows from (\ref{2.10}) that
\begin{equation}\label{2.12}
|\grg(\bfalp)|\ll H^{1+\eps}Y^{1/2+\eps}.
\end{equation}
Thus, we infer from (\ref{2.2}) that
\begin{align*}
\int_\grm \int_0^1\int_0^1 |\grg(\bfalp)|^6\d\alp_1 \d\alp_2 \d\alp_3& \ll 
\biggl( \sup_{\alp_3\in \grm}|\grg(\bfalp)|\biggr)^2 \oint |\grg(\bfalp)|^4\d\bfalp \\
&\ll (H^2Y)^{1+\eps}I_2(Y,H).
\end{align*}
An application of Lemma \ref{lemma2.1} consequently delivers the bound
\begin{align}
\int_\grm \int_0^1\int_0^1 |\grg(\bfalp)|^6\d\alp_1\d\alp_2\d\alp_3 &\ll (H^2Y)^{1+\eps} 
\left(H^3Y+(HY)^{2+\eps}\right) \notag \\
&\ll (HY)^{3\eps}\left( H^5Y^2+H^4Y^3\right) .\label{2.13}
\end{align}

\par Next, define the function $\Ups(\bet)$ for $\bet\in [0,1)$ by taking
$$\Ups(\bet)=(q+HY^2|q\bet -a|)^{-1},$$
when $\bet \in \grM(q,a)\subseteq \grM$, and by taking $\Ups(\bet)=0$ otherwise. Then it follows from 
(\ref{2.11}) that when $\alp_3\in \grM(q,a)\subseteq \grM$, one has
\begin{equation}\label{2.14}
|\grg(\bfalp)|\ll H^{1+\eps}Y^{1/2+\eps}+(HY)^{1+\eps}\Ups(\alp_3)^{1/2}.
\end{equation}
A comparison of (\ref{2.12}) and (\ref{2.14}) therefore leads via the argument underlying (\ref{2.13}) to 
the estimate
\begin{equation}\label{2.15}
\int_\grM \int_0^1\int_0^1 |\grg(\bfalp)|^6\d\alp_1\d\alp_2\d\alp_3\ll (HY)^{3\eps}
(H^5Y^2+H^4Y^3)+(HY)^{2+\eps}J_0,
\end{equation}
where
$$J_0=\sum_{1\le q\le Y}\sum^q_{\substack{a=1\\ (a,q)=1}}\int_{\grM(q,a)}\int_0^1
\int_0^1 \Ups(\alp_3)|\grg(\bfalp)|^4\d\alp_1\d\alp_2\d\alp_3 .$$
Write
$$\Psi(\bet)=\int_0^1\int_0^1 |\grg(\alp_1,\alp_2,\bet)|^4\d\alp_1 \d\alp_2 ,$$
so that $\Psi(\bet)$ has the Fourier expansion
$$\Psi(\bet)=\sum_{|n|\le 2HY^2}\psi_ne(\bet n),$$
with
$$\psi_n=\oint|\grg(\bfalp)|^4e(-n\alp_3)\d\bfalp .$$
Then it follows from \cite[Lemma 2]{Bru1988} that
\begin{equation}\label{2.16}
J_0\ll (HY^2)^{\eps-1}\biggl( Y\psi_0+\sum_{n\ne 0}|\psi_n|\biggr) .
\end{equation}

\par On the one hand, we find from (\ref{2.2}) that 
$$\psi_0=\oint |\grg(\bfalp)|^4\d\bfalp =I_2(Y,H).$$
Meanwhile, one has
$$\sum_{n\ne 0}|\psi_n|\le \sum_{n\in \dbZ}\oint |\grg(\bfalp)|^4e(-n\alp_3)\d\bfalp 
=\int_0^1\int_0^1 |\grg(\alp_1,\alp_2,0)|^4\d\alp_1\d\alp_2.$$
We therefore deduce from (\ref{2.16}), Lemma \ref{lemma2.1} and orthogonality that
\begin{equation}\label{2.17}
J_0\ll (HY^2)^{\eps-1}(H^3Y^2+H^2Y^3+J_1),
\end{equation}
where $J_1$ denotes the number of integral solutions of the simultaneous equations
\begin{equation}\label{2.18}
\left.\begin{aligned}
h_1y_1+h_2y_2&=h_3y_3+h_4y_4\\
h_1+h_2&=h_3+h_4
\end{aligned}
\right\} ,
\end{equation}
with $0<|h_i|\le H$ and $1\le y_i\le Y$ $(1\le i\le 4)$. A crude estimate for $J_1$ is obtained by simply 
observing that for each fixed choice of $h_i$ and $y_i$ $(1\le i\le 3)$, the equations (\ref{2.18}) uniquely 
determine $h_1$, and hence also $y_1$. Thus we find that $J_1\ll (HY)^3$, so that (\ref{2.17}) yields 
the bound $J_0\ll H^{2+\eps}Y^{1+\eps}$. On substituting this estimate into (\ref{2.15}), we therefore 
deduce that
$$\int_\grM \int_0^1\int_0^1 |\grg(\bfalp)|^6\d\alp_1\d\alp_2\d\alp_3 \ll 
(HY)^\eps (H^5Y^2+H^4Y^3).$$
Since $\grM\cup \grm=[0,1)$, it therefore follows from (\ref{2.13}) that
$$\oint |\grg(\bfalp)|^6\d\bfalp \ll (HY)^\eps (H^5Y^2+H^4Y^3),$$
and the conclusion of the lemma follows at once on recalling (\ref{2.2}).
\end{proof}

\section{New mean value estimates for exponential sums} The mean value estimates for $\grg(\bfalp)$ 
obtained in the previous section can be applied to convert estimates associated with the cubic case of 
Vinogradov's mean value theorem into estimates for mean values of the exponential sum $\grF(\alp,\bet)$ 
defined in (\ref{1.2}). In this section we discuss this conversion, and hence establish the bounds recorded 
in Theorem \ref{theorem1.3}. We recall the exponential sum $\grg(\bfalp;Y,H)$ defined in (\ref{2.1}), and 
write
$$\grg^*(\bfalp;Y,H)=\grg(\alp_1,2\alp_2,3\alp_3;Y,H).$$
In addition, we define the exponential sum $\grh(\bfalp)=\grh(\bfalp;X)$ by putting
$$\grh(\bfalp;X)=\sum_{1\le x\le X}e(\alp_1x+\alp_2x^2+\alp_3x^3).$$
Then, with the standard notation associated with Vinogradov's mean value theorem in mind, we put
\begin{equation}\label{3.1}
J_{s,3}(X)=\oint |\grh(\bfalp;X)|^{2s}\d\bfalp .
\end{equation}
Finally, we recall the definition of the mean value $T_s(X)$ from (\ref{1.3}).

\begin{lemma}\label{lemma3.1} When $s$ is a natural number, one has
$$T_s(X)\ll J_{s,3}(X)+X^{-1}\oint |\grh(\bfalp;2X)|^{2s}\grg^*(-\bfalp;X,sX)\d\bfalp .$$
\end{lemma}

\begin{proof} Define $\del_j$ to be $1$ when $j=1$, and to be $0$ otherwise. Our starting point for the 
proof of this lemma is the observation that, by orthogonality, the mean value $T_s(X)$ counts the number 
of integral solutions of the system of equations
\begin{equation}\label{3.2}
\sum_{i=1}^s(x_i^j-y_i^j)=\del_jh\quad (1\le j\le 3),
\end{equation}
with $1\le x_i,y_i\le X$ $(1\le i\le s)$ and $|h|\le sX$. Here, the constraint on
\begin{equation}\label{3.3}
\sum_{i=1}^s(x_i-y_i)
\end{equation}
imposed by the linear equation in (\ref{3.2}) is redundant, since the range for $h$ automatically 
accommodates all possible values of (\ref{3.3}) within (\ref{3.2}).\par

Let $T_s^*(X)$ denote the number of integral solutions of the system (\ref{3.2}) counted by $T_s(X)$ in 
which $h\ne 0$. Then on considering the underlying Diophantine systems, we see that
\begin{equation}\label{3.4}
T_s(X)=T_s^*(X)+J_{s,3}(X).
\end{equation}
Next, we consider the effect of shifting all of the variables by an integer $z$. By applying the binomial 
theorem, one finds that $\bfx,\bfy$ is a solution of (\ref{3.2}) if and only if $\bfx,\bfy$ is a solution of the 
system
\begin{equation*}
\sum_{i=1}^s\left( (x_i+z)^j-(y_i+z)^j\right)=jz^{j-1}h\quad (1\le j\le 3).
\end{equation*}

\par We therefore infer that for each fixed integer $z$ with $1\le z\le X$, the mean value $T_s^*(X)$ is 
bounded above by the number of integral solutions of the system
$$\sum_{i=1}^s\left( u_i^j-v_i^j\right)=jz^{j-1}h\quad (1\le j\le 3),$$
with $1\le u_i,v_i\le 2X$ $(1\le i\le s)$ and $0<|h|\le sX$. Thus, on recalling the definition (\ref{2.1}) and 
applying orthogonality, one finds that
\begin{align*}
\sum_{1\le z\le X}T_s^*(X)&\le \sum_{1\le z\le X}\oint |\grh(\bfalp;2X)|^{2s}\sum_{0<|h|\le sX}
e(-h\alp_1-2hz\alp_2-3hz^2\alp_3)\d\bfalp \\
&=\oint |\grh(\bfalp;2X)|^{2s}\grg^*(-\bfalp;X,sX)\d\bfalp .
\end{align*}
We therefore arrive at the relation
$$T_s^*(X)\ll X^{-1}\oint |\grh(\bfalp;2X)|^{2s}\grg^*(-\bfalp;X,sX)\d\bfalp .$$
The conclusion of the lemma follows by combining this estimate with (\ref{3.4}).
\end{proof}

We are now equipped to establish estimates for the moments of $\grF(\alp,\bet)$.

\begin{proof}[The proof of Theorem \ref{theorem1.3}] We first establish the estimate 
$T_4(X)\ll X^{4+\eps}$. We therefore apply Lemma \ref{lemma3.1}, with $s=4$, in combination with 
H\"older's inequality. On recalling (\ref{2.2}) and (\ref{3.1}), we obtain 
$T_4(X)\ll J_{4,3}(X)+X^{-1}I_1$, where 
$$I_1=\bigl(J_{5,3}(2X)\bigr)^{1/2}\bigl(J_{6,3}(2X)\bigr)^{1/4}\bigl(I_2(X,4X)\bigr)^{1/4}.$$
But \cite[Theorem 1.1]{Woo2014} establishes the main conjecture in the cubic case of Vinogradov's mean 
value theorem, and thus
\begin{equation}\label{3.5}
J_{s,3}(2X)\ll X^{s+\eps}\quad (1\le s\le 6).
\end{equation}
Meanwhile, Lemma \ref{lemma2.1} establishes the bound $I_2(X,4X)\ll X^{4+\eps}$. Hence
$$T_4(X)\ll X^{4+\eps}+X^{-1}(X^{5+\eps})^{1/2}(X^{6+\eps})^{1/4}
(X^{4+\eps})^{1/4}\ll X^{4+\eps}.$$
This confirms the first estimate of Theorem \ref{theorem1.3} in the case $s=4$. Meanwhile, when 
$1\le s\le 4$, an application of H\"older's inequality provides the bound 
$T_s(X)\le \bigl( T_4(X)\bigr)^{s/4}\ll X^{s+\eps}$, yielding the first estimate of the theorem in full.\par

We next take $s=5$, and apply Lemma \ref{lemma3.1} together with H\"older's inequality. Again recalling 
(\ref{2.2}) and (\ref{3.1}), we obtain $T_5(X)\ll J_{5,3}(X)+X^{-1}I_2$, where
$$I_2=\bigl( J_{6,3}(2X)\bigr)^{5/6}\bigl( I_3(X,5X)\bigr)^{1/6}.$$
Making use of Lemma \ref{lemma2.2} together with the main conjecture (\ref{3.5}) once again, we 
conclude that
\begin{equation}\label{3.6}
T_5(X)\ll X^{5+\eps}+X^{-1}(X^{6+\eps})^{5/6}(X^{7+\eps})^{1/6}\ll X^{31/6+\eps}.
\end{equation}
This completes the proof of the second estimate of the theorem.\par

The final estimate of the theorem is confirmed via the Hardy-Littlewood method. When $q\in \dbN$ and 
$a_2,a_3\in \dbZ$, denote a typical major arc by
$$\grM(q,a_2,a_3)=\{ \bfalp \in [0,1)^2:|q\alp_i-a_i|<(18X^{i-1})^{-1}\ (i=2,3)\}.$$
We take $\grM$ to be the union of these major arcs with $0\le a_i\le q\le P$ $(i=2,3)$ and 
$(q,a_2,a_3)=1$, and then put $\grm=[0,1)^2\setminus \grM$. Then it follows from 
\cite[Lemma 9.2]{Woo1991b} that whenever $t>9$, one has
\begin{equation}\label{3.7}
\int_\grM |\grF(\alp,\bet)|^t\d\alp\d\bet \ll X^{t-5}.
\end{equation}
Meanwhile, the argument leading to \cite[equation (7.9)]{Woo1991b} shows that
$$\sup_{(\alp,\bet)\in \grm}|\grF(\alp,\bet)|\ll X^{3/4+\eps}.$$
We consequently deduce from (\ref{3.6}) that for $t>10$, one has
$$\int_\grm |\grF(\alp,\bet)|^t\d\alp\d\bet \ll (X^{3/4+\eps})^{t-10}T_5(X)\ll 
X^{t-5+(32-3t)/12+t\eps}.$$
By combining this estimate with (\ref{3.7}), we conclude that whenever $t>32/3$ and $\eps$ is 
sufficiently small, then
$$T_t(X)=\int_\grm|\grF(\alp,\bet)|^t\d\alp\d\bet +\int_\grM|\grF(\alp,\bet)|^t\d\alp\d\bet \ll X^{t-5}.
$$
This completes the proof of the final estimate of the theorem.
\end{proof}

\section{Preliminary simplification of the diagonal equations} In the remainder of this memoir we focus on 
the system of equations (\ref{1.1}) and seek to prove Theorem \ref{theorem1.1}. Our application of the 
Hardy-Littlewood method to this system is simplified by some preliminary observations, much of the 
necessary work having been accomplished previously in \cite[\S6]{Woo1991b} and \cite[\S3]{Woo1998}. 
We start by showing that the conditions of Theorem \ref{theorem1.1} permit us to assume that (\ref{1.1}) 
has non-singular real solutions, and also non-singular $p$-adic solutions for each prime $p$. Thus, in our 
application of the circle method, we can expect both the singular series and singular integral to be 
non-zero.

\begin{lemma}\label{lemma4.1}
Suppose that conditions (a), (b) and (c) of Theorem \ref{theorem1.1} hold for the system of equations 
(\ref{1.1}). Then, either:
\item{(i)} the system has a real solution $(\bfx,\bfy,\bfz)=\bftet$ with the property that no $\tet_i$ is 
zero, and for which, locally, there is an $(s-2)$-dimensional subspace $\calS$ of positive $(s-2)$-volume in 
the neighbourhood of $\bftet$ on which $\Tet=\Phi=0$, or else
\item{(ii)} the system has a non-zero rational solution.
\end{lemma}

\begin{proof} The conclusion of the lemma follows by combining 
\cite[Lemmata 6.1 and 6.2]{Woo1991b}, just as in \cite[Lemma 3.1]{Woo1998}.
\end{proof}

Let $M(q)$ denote the number of solutions of the simultaneous congruences 
$\Tet(\bfx,\bfy)\equiv \Phi(\bfx,\bfz)\equiv 0\mmod{q}$ with $(\bfx,\bfy,\bfz)\in (\dbZ/q\dbZ)^s$.

\begin{lemma}\label{lemma4.2}
Suppose that $s\ge 11$ and the conditions (a), (b) and (c) of Theorem \ref{theorem1.1} hold for the 
system of equations (\ref{1.1}). Then, either:
\item{(i)} for each rational prime $p$, there is a natural number $w=w(p)$ with the property that for all 
$t\ge w$, one has $M(p^t)\ge p^{(t-w)(s-2)}$, or else
\item{(ii)} the system has a non-zero rational solution.
\end{lemma}

\begin{proof} This is \cite[Lemma 6.7]{Woo1991b}. We note that the conclusion (i) follows from the 
existence of a non-singular $p$-adic solution of the system (\ref{1.1}).
\end{proof}

Solubility is easily established when there are many vanishing coefficients.

\begin{lemma}\label{lemma4.3}
Suppose that $s\ge 11$ and the conditions (a), (b) and (c) of Theorem \ref{theorem1.1} hold for the 
system of equations (\ref{1.1}). Then the system has a non-zero solution in rational integers when either 
$m\ge 6$ or $n\ge 4$.
\end{lemma}

\begin{proof} The respective conclusions follow from \cite[Lemmata 6.3 and 6.5]{Woo1991b}.
\end{proof}

In view of Lemma \ref{lemma4.3}, we may assume henceforth that $s\ge 11$, $0\le m\le 5$ and 
$0\le n\le 3$, whence $l=s-m-n\ge s-8$. Also, by Lemma \ref{lemma4.1} together with the homogeneity 
of the system (\ref{1.1}), we may suppose that the latter equations have a non-singular real solution 
$(\bfx,\bfy,\bfz)=(\bfxi,\bfeta,\bfzet)=\bftet$ with the property that $0<|\tet_i|<\tfrac{1}{2}$ 
$(1\le i\le s)$. Since whenever necessary the $a_i$ can be replaced by $-a_i$ by interchanging $x_i$ and 
$-x_i$, and similarly $c_j$ may be replaced by $-c_j$ by interchanging $y_j$ and $-y_j$, we may 
suppose without loss that in fact $\tet_i>0$ $(1\le i\le s)$. Finally, as a consequence of Lemma 
\ref{lemma4.2}, we may suppose that for every rational prime $p$, there is a natural number $w=w(p)$ 
with the property that for all $t\ge w$, one has $M(p^t)\ge p^{(t-w)(s-2)}$. The latter bound also 
holds whenever (\ref{1.1}) has a non-singular $p$-adic solution for each prime $p$.\par

Our initial simplifications complete, we now record some notation to assist in our later deliberations. Let 
$P$ be a positive number sufficiently large in terms of $\eps$, $\bfa$, $\bfb$, $\bfc$, $\bfd$ and $\bftet$, 
and let $\alp_i$ $(i=2,3)$ be real variables. Also, write
$$t=\max_{i,j,k}\{ |a_i|,|b_i|,|c_j|,|d_k|\}.$$
We take $\eta$ to be a positive number sufficiently small in terms of $\eps$, put $R=P^\eta$, and then 
define $\calA(P,R)$ via (\ref{1.5}). We define the exponential sums
\begin{align*}
f_i(\alp_2,\alp_3)&=\sum_{\frac{1}{2}\xi_iP<x\le 2\xi_iP}e(a_i\alp_3x^3+b_i\alp_2x^2)\quad 
(1\le i\le l),\\
g_j(\alp_3)&=\sum_{\frac{1}{2}\eta_jP<y\le 2\eta_jP}e(c_j\alp_3y^3)\quad (1\le j\le m),\\
h_k(\alp_2)&=\sum_{\frac{1}{2}\zet_kP<z\le 2\zet_kP}e(d_k\alp_2z^2)\quad (1\le k\le n),
\end{align*}
and, when we wish to make use of analogous exponential sums in which the variable of summation is 
restricted to lie in the set of smooth numbers $\calA(P,R)$, we decorate this notation with a tilde. 
Thus, the exponential sum $\ftil_i(\alp_2,\alp_3)$ denotes such a smooth Weyl sum. For the sake 
of concision, we abbreviate
$$\text{$|f_i(\bfalp)|$ to $f_i$},\quad \text{$|g_j(\alp_3)|$ to $g_j$},\quad \text{and}\quad 
\text{$|h_k(\alp_2)|$ to $h_k$},$$
with similar conventions for other generating functions. We also write
$$f(\bfalp;X)=\sum_{1\le x\le X}e(\alp_3 x^3+\alp_2x^2)\quad \text{and}\quad 
g(\alp;X)=\sum_{1\le x\le X}e(\alp x^3).$$

\par Our aim is to estimate the number $R(P)$ of solutions of the Diophantine system (\ref{1.1}) in 
rational integers $x_i,y_j,z_k$ satisfying the conditions
\begin{equation}\label{4.1}
\tfrac{1}{2}\bfxi P<\bfx\le 2\bfxi P,\quad \tfrac{1}{2}\bfeta P<\bfy\le 2\bfeta P,\quad 
\tfrac{1}{2}\bfzet P<\bfz\le 2\bfzet P.
\end{equation}
Recall that there is no loss of generality in assuming in our proof of Theorem \ref{theorem1.1} that, in 
general, we have $0\le m\le 5$ and $0\le n\le 3$. The technical difficulties associated with our application 
of the Hardy-Littlewood method force us to divide systems of the shape (\ref{1.1}) into three classes:
\begin{enumerate}
\item[(A)] $m=n=0$ or $n\in \{1,2\}$,
\item[(B)] $1\le m\le 5$ and $n\in \{0,3\}$,
\item[(C)] $m=0$ and $n=3$.
\end{enumerate}
We adopt a different strategy for each class. For systems of type A, we obtain an asymptotic formula 
for $R(P)$. For systems of type B we instead count the number $R^*(P)$ of solutions of (\ref{1.1}) 
subject to (\ref{4.1}), and in addition constrained by the condition $y_j\in \calA(P,R)$ 
$(1\le j\le m)$. Finally, for systems of type C we count the number $R^\dagger (P)$ of solutions of 
(\ref{1.1}) subject to (\ref{4.1}), but now constrained by the condition $x_l\in \calA(P,R)$.\par

Plainly, one has $R(P)\ge R^*(P)$ and $R(P)\ge R^\dagger (P)$, and so if we show in the respective cases 
that $R(P)$, or $R^*(P)$, or $R^\dagger (P)$, grows in proportion to $P^{s-3}$ as 
$P\rightarrow \infty$, then in all cases we will be able to conclude that $R(P)\gg P^{s-3}$. This 
establishes the conclusion of Theorem \ref{theorem1.2}, and hence, in view of our earlier discussion, also 
the conclusion of Theorem \ref{theorem1.1}.\par

Next, we must describe the apparatus required for our application of the Hardy-Littlewood method. Write
$$\calF(\bfalp)=\prod_{i=1}^lf_i(\alp_2,\alp_3)\prod_{j=1}^mg_j(\alp_3)\prod_{k=1}^nh_k(\alp_2).$$
Also, denote by $\calF^*(\bfalp)$ the analogous generating function in which $g_j$ is decorated with a 
tilde for $1\le j\le m$, and $\calF^\dagger(\bfalp)$ that in which instead $f_l$ is decorated with a tilde.
Then it follows from orthogonality that
$$R(P)=\oint \calF(\bfalp)\d\bfalp,$$
and likewise when this relation is adorned with asterisks or obelisks.\par

Finally, we describe the Hardy-Littlewood dissection. Let $Q$ be a real number with $1\le Q\le P$, and put 
$\Xi_i=18tP^i$ $(i=2,3)$. Then, when $0\le r_i\le q\le Q$ $(i=2,3)$ and $(q,r_2,r_3)=1$, we denote a 
typical major arc by
$$\grM(q,\bfr;Q)=\{\bfalp \in [0,1)^2:|q\alp_i-r_i|\le Q\Xi_i^{-1}\ (i=2,3)\}.$$
Note that the arcs $\grM(q,\bfr)$ are disjoint whenever $1\le Q\le P$. Let $\grM(Q)$ be the union of these 
arcs $\grM(q,\bfr;Q)$, and put $\grm(Q)=[0,1)^2\setminus \grM(Q)$. For the sake of convenience, we 
put $\grM=\grM(P^{3/4})$ and $\grm=\grm(P^{3/4})$. Also, when $0\le r_i\le q\le Q$ $(i=2,3)$ and 
$(q,r_2,r_3)=1$, we denote a typical inhomogeneous major arc by
$$\grP(q,\bfr;Q)=\{ \bfalp\in[0,1)^2:|\alp_i-r_i/q|\le Q\Xi_i^{-1}\ (i=2,3)\} .$$
Let $\grP(Q)$ be the union of these arcs $\grP(q,\bfr;Q)$, and put $\grp(Q)=[0,1)^2\setminus \grP(Q)$. 
With $\del=10^{-6}$, we then put $\grP=\grP(P^{30\del})$ and $\grp=\grp(P^{30\del})$.\par

Henceforth, implicit constants in the notations of Landau and Vinogradov will depend at most on $s$, 
$\eps$, $\eta$, $\bfa$, $\bfb$, $\bfc$, $\bfd$ and $\bftet$, unless stated otherwise.

\section{Auxiliary estimates for systems of type A}
Our initial focus is on estimating $R(P)$ when $m=n=0$ or $n\in \{1,2\}$. We begin in this section by 
introducing several auxiliary mean value estimates, as well as an estimate of Weyl-type, useful both in 
estimating the minor arc contribution for systems of type A, as well as in later sections.

\begin{lemma}\label{lemma5.1} For all $i,j,k$, one has
\begin{enumerate}
\item[(i)] ${\displaystyle{\oint f_i^{32/3}\d\bfalp \ll P^{17/3+\eps}}}$,
\item[(ii)] ${\displaystyle{\oint f_i^8h_k^2\d\bfalp \ll P^{5+\eps}}}$,
\item[(iii)] ${\displaystyle{\oint f_i^4g_j^4h_k^2\d\bfalp \ll P^{5+\eps}}}$,
\item[(iv)] ${\displaystyle{\oint f_i^4g_j^8\d\bfalp \ll P^{7+\eps}}}$,
\item[(v)] ${\displaystyle{\oint g_j^8h_k^4\d\bfalp \ll P^{7+\eps}}}$.
\end{enumerate}
\end{lemma}

\begin{proof} We begin with the estimate (i). The definition of $f_i(\bfalp)$ implies that
$$f_i(\bfalp)=f(\bfalp;2\xi_iP)-f(\bfalp;\xi_iP/2),$$
and so whenever $w>16/3$, it follows from Theorem \ref{theorem1.3} that
$$\oint f_i^{2w}\d\bfalp \ll T_w(2\xi_iP)+T_w(\xi_iP/2)\ll P^{2w-5}.$$
Thus, we deduce from H\"older's inequality that when $\nu=\eps/5$, one has
$$\oint f_i^{32/3}\d\bfalp \ll \Bigl( \oint  f_i^{32(1+\nu)/3}\d\bfalp\Bigr)^{1/(1+\nu)}\ll 
P^{17/3+\eps}.$$

\par We turn next to the estimate (ii). By orthogonality, the mean value $I$ in question is bounded above 
by the number of integral solutions of the system
\begin{equation}
\left.\begin{aligned}
a_i\sum_{u=1}^4(x_u^3-y_u^3)&=0\\
b_i\sum_{u=1}^4(x_u^2-y_u^2)&=d_k(x_0^2-y_0^2)
\end{aligned}
\right\},\label{5.1}
\end{equation} 
with $1\le \bfx,\bfy\le P$. Denote by $I_0$ the number of these solutions with $x_0=y_0$, and $I_1$ the 
corresponding number with $x_0\ne y_0$. Then, on considering the underlying systems of Diophantine 
equations, it follows via Theorem \ref{theorem1.3} that $I_0\le PT_4(P)\ll P^{5+\eps}$. Suppose next 
that $\bfx,\bfy$ is a solution of (\ref{5.1}) counted by $I_1$. By applying Hua's 
lemma\footnote{See \cite[Lemma 2.5]{Vau1997}.} to the cubic equation in (\ref{5.1}), one sees that the 
number $I_2$ of choices for $x_u,y_u$ $(1\le u\le 4)$ satisfies $I_2=O(P^{5+\eps})$. Fix any one such 
choice. Since $x_0\ne y_0$, the integer
$$N=b_i\sum_{u=1}^4(x_u^2-y_u^2)$$
is fixed and non-zero. Both $x_0-y_0$ and $x_0+y_0$ are divisors of $N$, and so there are $O(N^\eps)$ 
possible choices for $x_0$ and $y_0$. Hence $I_1\ll N^\eps I_2\ll P^{5+4\eps}$. On combining these 
estimates, we obtain the desired bound $I=I_0+I_1\ll P^{5+\eps}$.\par

The argument for part (iii) is similar. By orthogonality, the mean value $I$ in question is bounded above 
by the number of integral solutions of the system
\begin{equation}\label{5.2}
\left.\begin{aligned}
&a_i\sum_{u=1}^2(x_u^3-y_u^3)&+c_j\sum_{u=3}^4(x_u^3-y_u^3)&=0\\ 
&b_i\sum_{u=1}^2(x_u^2-y_u^2)&&=d_k(x_0^2-y_0^2)
\end{aligned}\right\},
\end{equation}
with $1\le \bfx,\bfy\le P$. Denote by $I_3$ the number of these solutions with $x_0=y_0$, and $I_4$ the 
corresponding number with $x_0\ne y_0$. Suppose first that $\bfx,\bfy$ is a solution of (\ref{5.2}) 
counted by $I_3$, so that $x_0=y_0$. By applying Hua's lemma to the quadratic equation in (\ref{5.2}), 
one finds that the number $I_5$ of possible choices for $x_1,x_2,y_1,y_2$ satisfies 
$I_5=O(P^{2+\eps})$. Fix any one such choice. Then it follows from the cubic equation in (\ref{5.2}) 
that $c_j(x_3^3+x_4^3-y_3^3-y_4^3)=N$, where $N$ is now the fixed integer 
$-a_i(x_1^3+x_2^3-y_1^3-y_2^3)$. Hence, by the triangle inequality in combination with Hua's lemma, 
one see that the number $I_6$ of choices for $x_3,x_4,y_3,y_4$ satisfies $I_6=O(P^{2+\eps})$. Thus 
we conclude that $I_3\ll P(P^{2+\eps})^2\ll P^{5+2\eps}$. The argument required to bound $I_4$ is 
identical with that applied in case (ii), and thus we discern that $I_4\ll P^{5+\eps}$. By combining these 
estimates, we obtain the bound $I=I_3+I_4\ll P^{5+\eps}$ asserted in case (iii).\par

By orthogonality, the mean value $I$ in case (iv) is bounded above by the number of integral solutions of 
the system
\begin{equation}\label{5.3}
\left.\begin{aligned}
a_i\sum_{u=1}^2(x_u^3-y_u^3)&=c_j\sum_{u=3}^6(x_u^3-y_u^3)\\
b_i\sum_{u=1}^2(x_u^2-y_u^2)&=0
\end{aligned}
\right\},
\end{equation}
with $1\le \bfx,\bfy\le P$. Suppose that $\bfx,\bfy$ is a solution of (\ref{5.3}) counted by $I$. By applying 
Hua's lemma to the quadratic equation in (\ref{5.3}), one sees that the number $I_7$ of possible choices 
for $x_1,x_2,y_1,y_2$ satisfies $I_7=O(P^{2+\eps})$. Fix any one such choice. We now apply the 
triangle inequality in combination with Hua's lemma to the cubic equation of (\ref{5.3}) in a manner 
similar to that of the previous case. Thus, the number $I_8$ of possible choices for $x_u,y_u$ 
$(3\le u\le 6)$ satisfies $I_8=O(P^{5+\eps})$. We therefore conclude that 
$I\ll (P^{2+\eps})(P^{5+\eps})=P^{7+2\eps}$, confirming the estimate asserted in case (iv) of the 
lemma.\par

The mean value $I$ in case (v) counts the number of integral solutions of
$$d_k\sum_{u=1}^2(x_u^2-y_u^2)=c_j\sum_{u=3}^6(x_u^3-y_u^3)=0,$$
with $1\le \bfx,\bfy\le P$. A comparison with (\ref{5.3}) reveals that an argument identical to that applied 
in case (iv) may be used to confirm the bound $I\ll P^{7+\eps}$. This completes the proof of the lemma.
\end{proof}

Finally, we recall a Weyl estimate for $f_i(\bfalp)$ sensitive to both $\alp_2$ and $\alp_3$.

\begin{lemma}\label{lemma5.2}
Suppose that $Q\le P^{3/4}$. Then for all $i$, one has
$$\sup_{\bfalp\in \grp(Q)}|f_i(\bfalp)|\le \sup_{\bfalp\in \grm(Q)}|f_i(\bfalp)|\ll 
P^{1+\eps}Q^{-1/3}.$$
\end{lemma}

\begin{proof} Let $\tau>0$, so that $\left( P/(P^{1+\tau}Q^{-1/3})\right)^{4+\eps}<P$. If 
$|f_i(\bfalp)|\ge P^{1+\tau}Q^{-1/3}$, then we see from \cite[Theorem 5.1]{Bak1986} that there exist 
$q\in \dbN$ and $\bfr\in \dbZ^2$ such that $(q,r_2,r_3)=1$, $q<P^{-\tau}Q$ and 
$|q\alp_j-r_j|<QP^{-\tau-j}$ $(j=2,3)$. Thus $\bfalp\in \grM(Q)$, and so it follows that whenever 
$\bfalp\in \grm(Q)$, then $|f_i(\bfalp)|<P^{1+\tau}Q^{-1/3}$. Since $\grp(Q)\subseteq \grm(Q)$, the 
conclusion of the lemma follows.
\end{proof}

\section{The minor arc contribution for systems of type A} We now estimate the contribution of the 
minor arcs $\grp$ within the integral giving $R(P)$ for systems of type A. Thus we may suppose that 
either $m=n=0$ or else $0\le m\le 5$ and $n\in \{1,2\}$. Here and later we make use of the inequality
\begin{equation}\label{6.1}
|z_1z_2\ldots z_n|\le |z_1|^n+\ldots +|z_n|^n.
\end{equation}

\begin{lemma}\label{lemma6.1} One has 
$\displaystyle{\int_\grp |\calF(\bfalp)|\d\bfalp \ll P^{s-5-\del}}$.
\end{lemma}

\begin{proof} On recalling the definition of $\calF(\bfalp)$ and applying (\ref{6.1}), we deduce that 
for some integers $i$, $j$ and $k$, one has
\begin{equation}\label{6.2}
\int_\grp |\calF(\bfalp)|\d\bfalp \ll \int_\grp f_i^lg_j^mh_k^n\d\bfalp .
\end{equation}
In view of our hypotheses concerning $m$ and $n$, by repeated application of (\ref{6.1}), as in the 
proof of \cite[Lemma 7.3]{Woo1991b}, one obtains the bound
$$\int_\grp f_i^lg_j^mh_k^n\d\bfalp \ll \int_\grp \left( f_i^s+f_i^{s-2}h_k^2+f_i^{s-7}g_j^5h_k^2
+f_i^{s-6}g_j^5h_k\right) \d\bfalp .$$ 
Thus it follows from (\ref{6.2}) that
\begin{equation}\label{6.3}
\int_\grp |\calF(\bfalp)|\d\bfalp \ll \Bigl( \sup_{\bfalp \in \grp}f_i\Bigr)^{s-65/6}
(I_{0,0}+I_{0,2}+I_{5,2}+I_{5,1}),
\end{equation}
where
$$I_{a,b}=\oint f_i^{65/6-a-b}g_j^ah_k^b\d\bfalp .$$

\par The trivial estimate $f_i\le P$ combines with Lemma \ref{lemma5.1}(i) to give
$$I_{0,0}\le P^{1/6}\oint f_i^{32/3}\d\bfalp \ll P^{35/6+\eps}.$$
In a similar manner, one finds from Lemma \ref{lemma5.1}(ii) that
$$I_{0,2}\le P^{5/6}\oint f_i^8h_k^2\d\bfalp \ll P^{35/6+\eps}.$$
Next, by applying first H\"older's inequality, and then Lemma \ref{lemma5.1}(i), (iii), (iv) and (v), 
one obtains the bound
\begin{align*}
I_{5,2}&\le \biggl( \oint f_i^{32/3}\d\bfalp \biggr)^{1/8}\biggl( \oint f_i^4g_j^4h_k^2\d\bfalp 
\biggr)^{1/2}\biggl( \oint f_i^4g_j^8\d\bfalp\biggr)^{1/8}\biggl( \oint g_j^8h_k^4\d\bfalp 
\biggr)^{1/4}\\
&\ll P^\eps (P^{17/3})^{1/8}(P^5)^{1/2}(P^7)^{1/8}(P^7)^{1/4}\ll P^{35/6+\eps}.
\end{align*}
Similarly, but now using Lemma \ref{lemma5.1}(i),(iii) and (iv), one finds that
\begin{align*}
I_{5,1}&\le \biggl( \oint f_i^{32/3}\d\bfalp \biggr)^{1/8}\biggl( \oint f_i^4g_j^4h_k^2\d\bfalp 
\biggr)^{1/2}\biggl( \oint f_i^4g_j^8\d\bfalp \biggr)^{3/8}\\
&\ll P^\eps (P^{17/3})^{1/8}(P^5)^{1/2}(P^7)^{3/8}\ll P^{35/6+\eps}.
\end{align*}
Finally, Lemma \ref{lemma5.2} supplies the bound
$$\sup_{\bfalp\in \grp} f_i\ll P^{1+\eps}(P^{30\del})^{-1/3}\ll P^{1-9\del}.$$
Thus, we conclude from (\ref{6.3}) that
$$\int_\grp |\calF(\bfalp)|\d\bfalp \ll (P^{1-9\del})^{s-65/6}P^{35/6+\eps}\ll P^{s-5-\del},$$
and the proof of the lemma is complete.
\end{proof}

\section{The major arc contribution for systems of type A} 
We next estimate the contribution of the major arcs within $R(P)$ for systems of type A, beginning 
with some additional notation. For each $i,j,k$, we write
$$S_{f,i}(q,\bfr)=\sum_{u=1}^qe_q(a_ir_3u^3+b_ir_2u^2),\quad 
v_{f,i}(\bfbet)=\int_{\xi_iP/2}^{2\xi_iP}e(a_i\bet_3\gam^3+b_i\bet_2\gam^2)\d\gam ,$$
$$S_{g,j}(q,\bfr)=\sum_{u=1}^qe_q(c_jr_3u^3),\quad v_{g,j}(\bfbet)=\int_{\eta_jP/2}^{2\eta_jP}
e(c_j\bet_3\gam^3)\d\gam ,$$
$$S_{h,k}(q,\bfr)=\sum_{u=1}^qe_q(d_mr_2u^2),\quad v_{h,k}(\bfbet)=
\int_{\zet_kP/2}^{2\zet_kP}e(d_m\bet_2\gam^2)\d\gam ,$$
where, as usual, we write $e_q(z)$ for $e^{2\pi iz/q}$. We then define
$$T(q,\bfr)=q^{-s}\prod_{i=1}^lS_{f,i}(q,\bfr)\prod_{j=1}^mS_{g,j}(q,\bfr)
\prod_{k=1}^nS_{h,k}(q,\bfr)$$
and
$$V(\bfbet)=\prod_{i=1}^lv_{f,i}(\bfbet)\prod_{j=1}^mv_{g,j}(\bfbet)\prod_{k=1}^nv_{h,k}(\bfbet).$$
We recall some estimates for these generating functions recorded in \cite{Woo1991b}.

\begin{lemma}\label{lemma7.1}
Suppose that $q\in \dbN$ and $\bfr\in \dbZ$ satisfy $(q,r_2,r_3)=1$. Then 
$$S_{f,i}(q,\bfr)\ll q^{2/3+\eps},\quad S_{g,j}(q,\bfr)\ll q^{2/3+\eps}(q,r_3)^{1/3},$$
$$S_{h,k}(q,\bfr)\ll q^{1/2+\eps}(q,r_2)^{1/2}.$$
When $p$ is a prime number with $(p,r_2,r_3)=1$, and $h\in \{1,2\}$, moreover, then 
$$S_{f,i}(p^h,\bfr)\ll p^{h/2},\quad S_{g,j}(p^h,\bfr)\ll p^{h/2}(p^h,r_3)^{1/2},$$
$$S_{h,k}(p^h,\bfr)\ll p^{h/2}(p^h,r_2)^{1/2}.$$
\end{lemma}

\begin{proof} The first batch of estimates follow from \cite[Theorem 7.1]{Vau1997}. Meanwhile, the 
first of the second batch follows from \cite[Corollary 2F of Chapter II]{Sch1976} in the case $h=1$, 
and from the argument of the proof of \cite[Theorem 7.1]{Vau1997} in the case $h=2$. The final two 
estimates are immediate from \cite[Lemmata 4.3 and 4.4]{Vau1997}. Implicit constants here may depend 
on the coefficients $a_i,b_i,c_j,d_k$.
\end{proof}

\begin{lemma}\label{lemma7.2}
One has
$$v_{f,i}(\bfbet)\ll P(1+P^2|\bet_2|+P^3|\bet_3|)^{-1/3},\quad v_{g,j}(\bfbet)\ll 
P(1+P^2|\bet_3|)^{-1/3},$$
$$v_{h,k}(\bfbet)\ll P(1+P^2|\bet_2|)^{-1/2}.$$
\end{lemma}

\begin{proof} The respective estimates follow from \cite[Theorem 7.3]{Vau1997}.
\end{proof}

We define the function $f_i^*(\bfalp)$ for $\bfalp\in [0,1)^2$ by taking
\begin{equation}\label{7.1}
f_i^*(\bfalp)=q^{-1}S_{f,i}(q,\bfr)v_{f,i}(\bfalp-\bfr/q),
\end{equation}
when $\bfalp\in \grM(q,\bfr;P)\subseteq \grM(P)$, and by taking $f_i^*(\bfalp)=0$ when 
$\bfalp \not \in \grM(P)$. We define $g_j^*(\bfalp)$ and $h_k^*(\bfalp)$ in an analogous manner.

\begin{lemma}\label{lemma7.3}
Suppose that $q\in \dbN$ and $\bfr\in \dbZ$ satisfy $(q,r_2,r_3)=1$. Then, when 
$\bfalp \in \grM(q,\bfr;P)\subseteq \grM(P)$, one has 
$$f_i(\bfalp)-f_i^*(\bfalp)\ll q^{2/3+\eps},\quad g_j(\alp_3)-g_j^*(\bfalp)\ll 
q^{2/3+\eps}(q,r_3)^{1/3},$$
$$h_k(\alp_2)-h_k^*(\bfalp)\ll q^{1/2+\eps}(q,r_2)^{1/2}.$$
\end{lemma}

\begin{proof}
This is immediate from \cite[Lemma 4.4]{Bak1986}.
\end{proof}

Next we introduce the incomplete singular series and integral
$$\grS(Q)=\sum_{1\le q\le Q}\underset{(q,r_2,r_3)=1}{\sum_{r_2=1}^q\sum_{r_3=1}^q}T(q,\bfr)\quad 
\text{and}\quad \grJ(Q)=\int_{-QP^{-3}}^{QP^{-3}}\int_{-QP^{-2}}^{QP^{-2}}V(\bfbet)\d\bfbet ,$$
and their completed counterparts
$$\grS=\sum_{q=1}^\infty \underset{(q,r_2,r_3)=1}{\sum_{r_2=1}^q\sum_{r_3=1}^q}T(q,\bfr)\quad 
\text{and}\quad \grJ=\iint_{\dbR^2}V(\bfbet)\d\bfbet .$$
The truncated singular integral $\grJ(Q)$ is easily estimated via Lemma \ref{lemma7.2}.

\begin{lemma}\label{lemma7.4} Provided that $s\ge 11$, $n\le 3$ and $m\le 5$, there is a positive 
constant $\calC$ with the property that $\grJ(Q)=\calC P^{s-5}+O(P^{s-5}Q^{-1/2})$.
\end{lemma}

\begin{proof} An application of Lemma \ref{lemma7.2} reveals that
$$V(\bfbet)\ll P^s(1+P^3|\bet_3|)^{-m/3}(1+P^2|\bet_2|)^{-n/2}(1+P^2|\bet_2|+P^3|\bet_3|)^{-l/3}.$$
Our hypotheses $s=l+m+n\ge 11$, $m\le 5$ and $n\le 3$ therefore ensure that
$$V(\bfbet)\ll P^s(1+P^3|\bet_3|)^{-11/6}(1+P^2|\bet_2|)^{-11/6}.$$
When $\bfbet$ lies outside the box $[-QP^{-2},QP^{-2}]\times [-QP^{-3},QP^{-3}]$, one has either 
$P^2|\bet_2|\ge Q$ or $P^3|\bet_3|\ge Q$. Hence we deduce that
$$\grJ-\grJ(Q)\ll P^sQ^{-1/2}\iint_{\dbR^2}(1+P^3|\bet_3|)^{-4/3}(1+P^2|\bet_2|)^{-4/3}\d\bfbet ,$$
so that $\grJ(Q)=\grJ+O(P^{s-5}Q^{-1/2})$.\par

In view of Lemma \ref{lemma4.1}, we may assume that the equations $\Tet=\Phi=0$ define an 
$(s-2)$-dimensional manifold $\calS$ in the box 
$\calB=[\tfrac{1}{2}\tet_1,2\tet_1]\times \ldots \times [\tfrac{1}{2}\tet_s,2\tet_s]$, passing through 
the point $\bftet$ and having positive $(s-2)$-volume. By making a change of variables one sees that
$$\grJ=P^{s-5}\iint_{\dbR^2}\int_\calB e(\bet_3\Tet (\bfgam)+\bet_2\Phi(\bfgam))\d\bfgam \d\bfbet .
$$
Applying Fourier's integral formula twice to the latter integral, in the form
$$\lim_{\lam\rightarrow \infty}\int_{-\Gam}^\Gam\int_{-\lam}^\lam V(\gam)e(\gam\ome)\d\ome \d 
\gam=V(0),$$
we therefore see that $\grJ=\calC P^{s-5}$, where $\calC>0$ is the $(s-2)$-volume of $\calS$ within 
$\calB$. Thus we deduce that $\grJ=\calC P^{s-5}$, and hence 
$\grJ(Q)=\calC P^{s-5}+O(P^{s-5}Q^{-1/2})$. This completes the proof of the lemma.
\end{proof}

Next we consider the truncated singular series $\grS(Q)$.

\begin{lemma}\label{lemma7.5} Provided that $s\ge 11$, $n\le 3$ and $m\le 5$, the singular series 
$\grS$ converges absolutely, one has $\grS>0$, and $\grS(Q)=\grS+O(Q^{-1/6})$.
\end{lemma}

\begin{proof} We begin by investigating the quantity
$$A(q)=\underset{(q,r_2,r_3)=1}{\sum_{r_2=1}^q\sum_{r_3=1}^q}|T(q,\bfr)|$$
when $q=p^h$ is a power of the prime number $p$. First, when $h\ge 2$, it follows from Lemma 
\ref{lemma7.1} and our hypotheses concerning $s$, $m$ and $n$ that
\begin{align}
A(p^h)&\ll (p^h)^{\eps-s/3}\sum_{r_2=1}^{p^h}\sum_{r_3=1}^{p^h}
p^{-nh/6}(p^h,r_2)^{n/2}(p^h,r_3)^{m/3}\notag \\
&\ll (p^h)^{\eps-s/3}\sum_{a=0}^h\sum_{b=0}^h\left( p^{h-a}(p^a)^{n/3}\right) 
\left( p^{h-b}(p^b)^{m/3}\right) \notag \\
&\ll (p^h)^{2+\eps-s/3}\sum_{a=0}^h\sum_{b=0}^hp^{2b/3}\ll h(p^h)^{\eps-1}.\label{7.2}
\end{align}
When $h=1$ we proceed similarly, applying Lemma \ref{lemma7.1} to obtain
\begin{align}
A(p)&\ll p^{-s/2}\underset{(p,r_2,r_3)=1}{\sum_{r_2=1}^p\sum_{r_3=1}^p}
(p,r_3)^{m/2}(p,r_2)^{n/2}\notag \\
&\ll p^{-s/2}(p^2+p^{1+m/2}+p^{1+n/2})\ll p^{-2}.\label{7.3}
\end{align}

\par The standard theory associated with singular series shows that $A(q)$ is multiplicative (compare 
\cite[Lemmata 10.4 and 10.5]{Woo1991b}). Then 
we deduce from (\ref{7.2}) and (\ref{7.3}) that for some constant $C$ depending only on $\bfa$, $\bfb$, 
$\bfc$, $\bfd$, one has
\begin{align*}
\sum_{q=1}^\infty q^{1/6}A(q)&\le \prod_p \Bigl( 1+Cp^{-11/6}+
C\sum_{h=2}^\infty hp^{-2h/3}\Bigr) \\
&\le \prod_p (1+C^2p^{-4/3})\ll 1.
\end{align*}
Both the absolute convergence of $\grS$ and the final conclusion of the lemma follow by applying this 
bound to show that
$$|\grS-\grS(Q)|\le \sum_{q>Q}(q/Q)^{1/6}A(q)\ll Q^{-1/6}.$$
Furthermore, the argument underlying the proof of \cite[Lemmata 10.8]{Woo1991b} combines with the 
bounds (\ref{7.2}) and (\ref{7.3}) to confirm that the quantity
$$\chi_p=\sum_{h=0}^\infty \underset{(p,r_2,r_3)=1}{\sum_{r_2=1}^{p^h}\sum_{r_3=1}^{p^h}}
T(p^h,\bfr)$$
satisfies $\chi_p\ll 1$, and that $\grS=\prod_p\chi_p$. In view of Lemma \ref{lemma4.2}, we are at 
liberty to assume that, for each prime $p$, there is a natural number $w=w(p)$ with the property that 
$M(p^t)\ge p^{(t-w)(s-2)}$ for $t\ge w$. The conclusion of \cite[Lemma 10.9]{Woo1991b} therefore 
establishes that $\chi_p>0$ for each prime number $p$, and hence that $\grS=\prod_p\chi_p>0$. This 
completes the proof of the lemma.
\end{proof}

We now complete the analysis of the major arcs for systems of type A. From Lemma \ref{lemma7.3} we 
see that when $\bfalp\in \grP(q,\bfr)\subseteq \grP$, one has
$$f_i(\bfalp)-f_i^*(\bfalp)\ll P^{31\del},\quad g_j(\alp_3)-g_j^*(\bfalp)\ll P^{31\del},\quad 
h_k(\alp_2)-h_k^*(\bfalp)\ll P^{31\del},$$
and hence
$$\calF(\bfalp)-T(q,\bfr)V(\bfalp-\bfr/q)\ll P^{s-1+31\del}.$$
The measure of the set of arcs $\grP$ is $O((P^{30\del})^5P^{-5})$, and thus we conclude that
$$\int_\grP \calF(\bfalp)\d\bfalp -\grS(P^{30\del})\grJ(P^{30\del})\ll (P^{s-1+31\del})(P^{150\del-5})
\ll P^{s-5-\del}.$$
We therefore infer from Lemmata \ref{lemma7.4} and \ref{lemma7.5} that
$$\int_\grP \calF(\bfalp)\d\bfalp =\grS \calC P^{s-5}+O(P^{s-5-\del}).$$
Since $[0,1)^2$ is the union of $\grP$ and $\grp$, we conclude from Lemma \ref{lemma6.1} that
\begin{equation}\label{7.4}
R(P)=\iint_{[0,1)^2}\calF(\bfalp)\d\bfalp =\calC \grS P^{s-5}+O(P^{s-5-\del}).
\end{equation}
Since $\calC \grS>0$, we deduce that $R(P)\gg P^{s-5}$. The conclusion of Theorem \ref{theorem1.2}, 
and hence also Theorem \ref{theorem1.1}, therefore follows for systems of type A.

\section{Auxiliary estimates for systems of type B} In the next phase of our argument, we focus on 
estimating $R^*(P)$ when $1\le m\le 5$ and $n\in\{0,3\}$. We again begin by introducing several 
auxiliary mean value estimates useful both here and later. The value of $\eta$ is chosen in accordance 
with the following lemma.

\begin{lemma}\label{lemma8.1}
Suppose that $\eta>0$ is sufficiently small. Then for all $j$, one has
$$\int_0^1\gtil_j^6\d\bet \ll P^{13/4-3\del}\quad \text{and}\quad \int_0^1\gtil_j^{31/4}\d\bet \ll 
P^{19/4}.$$
\end{lemma}

\begin{proof} The definition of $\gtil_j(\bet)$ implies that
$$\gtil_j(\bet)=\gtil(\bet;2\eta_jP)-\gtil(\bet;\eta_jP/2),$$
and so it follows from \cite[Theorem 1.2]{Woo2000} and \cite[Theorem 2]{BW2002} that
$$\int_0^1\gtil_j^u\d\bet \ll \max_{1\le X\le P}\int_0^1|\gtil(\bet;X)|^u\d\bet \ll P^{\mu_u},$$
where $\mu_6=3.2495$, and $\mu_u=u-3$ whenever $u\ge 7.691$. The conclusion of the lemma follows at 
once.
\end{proof}

\begin{lemma}\label{lemma8.2} For all $i,j,k$, one has
\begin{enumerate}
\item[(i)] $\displaystyle{\oint f_i^4\gtil_j^6\d\bfalp \ll P^{21/4-2\del}}$,
\item[(ii)] ${\displaystyle{\oint \gtil_j^6h_k^4\d\bfalp \ll P^{21/4-2\del}}}$,
\item[(iii)] ${\displaystyle{\oint \gtil_j^{31/4}h_k^{17/4}\d\bfalp \ll P^7}}$.
\end{enumerate}
\end{lemma}

\begin{proof} We begin with the estimate (i). By orthogonality, the mean value $I$ in question is 
bounded above by the number of integral solutions of the system
\begin{equation}\label{8.1}
\left.\begin{aligned}
a_i\sum_{u=1}^2(x_u^3-y_u^3)&=c_j\sum_{u=3}^5(x_u^3-y_u^3)\\
b_i\sum_{u=1}^2(x_u^2-y_u^2)&=0
\end{aligned}\right\},
\end{equation}
with $1\le \bfx,\bfy\le P$ and $x_u,y_u\in \calA(P,R)$ $(3\le u\le 5)$. Suppose that $\bfx,\bfy$ is a 
solution of (\ref{8.1}) counted by $I$. By applying Hua's lemma to the quadratic equation in 
(\ref{8.1}), one sees that the number $I_1$ of possible choices for $x_1,x_2,y_1,y_2$ satisfies 
$I_1=O(P^{2+\eps})$. Fix any one such choice. Then it follows from the cubic equation in (\ref{8.1}) 
that $c_j(x_3^3+x_4^3+x_5^3-y_3^3-y_4^3-y_5^3)=N$, where $N$ is the fixed integer 
$a_i(x_1^3+x_2^3-y_1^3-y_2^3)$. Hence, by the triangle inequality in combination with Lemma 
\ref{lemma8.1}, one sees that the number $I_2$ of choices for $x_u,y_u$ $(3\le u\le 5)$ satisfies 
$I_2=O(P^{13/4-3\del})$. Thus $I\ll (P^{2+\eps})(P^{13/4-3\del})\ll P^{21/4-2\del}$, confirming the 
estimate (i) asserted in the lemma.\par

We turn next to the estimate (ii). Here, Lemma 8.1 combines with Hua's lemma to show that the mean 
value in question is equal to
$$\int_0^1\gtil_j^6\d\alp_3\int_0^1h_k^4\d\alp_2\ll (P^{13/4-3\del})(P^{2+\eps}),$$
and hence the estimate (ii) follows.\par

In order to establish the estimate (iii), we begin by noting that a straightforward application of the 
circle method shows that whenever $u>4$, one has
\begin{equation}\label{11.1}
\int_0^1 h_k^u\d\alp_2 \ll P^{u-2}.
\end{equation}
Thus, Lemma \ref{lemma8.1} shows that the mean value now in question is equal to
\begin{align*}
\int_0^1 \gtil_j^{31/4}\d\alp_3\int_0^1 h_k^{17/4}\d\alp_2\ll (P^{19/4})(P^{9/4})\ll P^7.
\end{align*}
This completes the proof of the estimate (iii).
\end{proof}

We define the multiplicative function $\kap(q)=\kap_C(q)$ for prime powers $p^h$ by taking 
$\kap(p^h)$ to be $Cp^{-h/2}$, when $h\in \{1,2\}$, and to be $Cp^{\eps-h/3}$, when $h>2$.

\begin{lemma}\label{lemma8.3}
For all $i,j,k$, one has
\begin{enumerate}
\item[(i)] $\displaystyle{\sup_{\bfalp \in \grM(P)\setminus \grM(Q)}|f^*_i(\bfalp)|\ll PQ^{-1/4}
\quad (Q\le P)}$,
\item[(ii)] $\displaystyle{\oint |f_i^*|^u\d\bfalp \ll P^{u-5}\quad (u>7)}$,
\item[(iii)] $\displaystyle{\oint |f_i^*|^u\gtil_j^{31/4}\d\bfalp \ll P^{u+11/4}\quad 
(u>4)}$.
\item[(iv)] $\displaystyle{\oint |f_i^*h_k|^u\d\bfalp \ll P^{2u-5}}$ $(u>4)$.
\end{enumerate}
\end{lemma}

\begin{proof} We begin with the estimate (i). Suppose that $\bfalp\in \grM(q,\bfr)\subseteq \grM(P)$. 
Then on recalling (\ref{7.1}), we find from Lemmata \ref{lemma7.1} and \ref{lemma7.2} that
$$f_i^*(\bfalp)\ll P(q+P^2|q\alp_2-r_2|+P^3|q\alp_3-r_3|)^{\eps -1/3}.$$
When $\bfalp\not\in \grM(Q)$, one of the lower bounds $q>Q$, or $|q\alp_2-r_2|>QP^{-2}$, or 
$|q\alp_3-r_3|>QP^{-3}$ must hold, and hence $|f_i^*(\bfalp)|\ll PQ^{-1/4}$.\par

We next turn to the estimate (ii), and assume throughout that $u>7$. On recalling (\ref{7.1}), we 
deduce from Lemma \ref{lemma7.2} that
\begin{equation}\label{8.3}
\oint |f_i^*|^u\d\bfalp \ll \calI_u(P)W_u(P),
\end{equation}
where
$$\calI_u(P)=\int_{-1/2}^{1/2}\int_{-1/2}^{1/2}P^u(1+P^2|\bet_2|+P^3|\bet_3|)^{\eps -u/3}\d\bfbet ,$$
$$W_u(P)=\sum_{1\le q\le P}A_u(q)\quad \text{and}\quad A_u(q)=
\underset{(q,r_2,r_3)=1}{\sum_{r_2=1}^q\sum_{r_3=1}^q}|q^{-1}S_{f,i}(q,\bfr)|^u.$$
On the one hand,
$$\calI_u(P)\le P^u\int_{-\infty}^\infty \int_{-\infty}^\infty (1+P^2|\bet_2|)^{\eps-u/6}
(1+P^3|\bet_3|)^{\eps-u/6}\d \bfbet \ll P^{u-5}.$$
Meanwhile, when $p$ is prime, it follows from Lemma \ref{lemma7.1} that for $h\ge 2$, one has 
$A_u(p^h)\ll p^{2h}(p^h)^{\eps-u/3}\ll p^{-h/3-\nu}$, for some $\nu>0$, and for $h\in \{1,2\}$, 
instead $A_u(p^h)\ll p^{2h}(p^h)^{-u/2}\ll p^{-3h/2}$. Since the standard theory of singular series 
shows that $A(q)$ is multiplicative, we deduce that for some $C>0$, one has
$$W_u(P)\ll \prod_{p\le P}\left( 1+Cp^{-3/2}+Cp^{-1-\nu}\right) \ll 1.$$
The estimate (ii) is confirmed by substituting these estimates into (\ref{8.3}).\par

Considering next the estimate (iii), we may assume that $u>4$. From \cite[Lemma 10.4]{Woo1991b}, 
one finds that $S_{f,i}(q,\bfr)$ possesses the usual quasi-multiplicative property. Hence it follows 
from Lemma \ref{lemma7.1} that there is a number $C>0$ with the property that whenever 
$(q,r_2,r_3)=1$, then $S_{f,i}(q,\bfr)\ll \kap(q)$. Thus, we deduce from (\ref{7.1}) via Lemma 
\ref{lemma7.2} that whenever $\bfalp\in \grM(q,\bfr)\subseteq \grM(P)$, then
$$|f_i^*(\bfalp)|\ll \kap(q)P(1+P^2|\alp_2-r_2/q|)^{-1/3}.$$
Hence
\begin{equation}\label{8.4}
\oint|f_i^*|^u\gtil_j^{31/4}\d\bfalp \ll \calJ_u(P)\sum_{1\le q\le P}\sum_{r_2=1}^q\kap(q)^u,
\end{equation}
where
$$\calJ_u(P)=P^u\int_0^1\int_{-1/2}^{1/2}(1+P^2|\bet_2|)^{-u/3}|\gtil_j(\alp_3)|^{31/4}
\d\bet_2\d\alp_3.$$
Since $u>4$, on the one hand, one finds by means of Lemma \ref{lemma8.1} that
$$\calJ_u(P)\ll P^{u-2}\int_0^1\gtil_j^{31/4}\d\alp_3\ll P^{u+11/4}.$$
Provided that $C$ is chosen sufficiently large, on the other hand, it follows from the multiplicative 
property of $\kap(q)$ that there is a number $\nu>0$ such that
\begin{equation}\label{8.5}
\sum_{1\le q\le P}q\kap(q)^u\ll \prod_p(1+C^2p^{-1-\nu})\ll 1.
\end{equation}
The estimate (iii) follows by substituting these bounds into (\ref{8.4}).

Finally, we turn to the estimate (iv), and again assume that $u>4$. Then, as in the discussion leading to 
(\ref{8.4}), one sees that whenever $\bfalp\in \grM(q,\bfr)\subseteq \grM(P)$, there is a number $C>0$ 
such that
$$|f_i^*(\bfalp)|\ll \kap(q)P(1+P^3|\alp_3-r_3/q|)^{-1/3}.$$
Thus we find that whenever $u>4$, one has
\begin{equation}\label{11.2}
\oint |f_i^*|^uh_k^u\d\bfalp \ll \calJ_u(P)\sum_{1\le q\le P}\sum_{r_3=1}^q\kap(q)^u,
\end{equation}
where
$$\calJ_u(P)=P^u\int_{-1/2}^{1/2}\int_0^1 |h_k(\alp_2)|^u(1+P^3|\bet_3|)^{-u/3}\d\alp_2\d\bet_3.
$$
On the one hand, an estimate of the shape (\ref{11.1}) shows that
$$\calJ_u(P)\ll P^{u-3}\int_0^1h_k^u\d\alp_2\ll P^{2u-5}.$$
Provided that $C$ is chosen sufficiently large, on the other hand, it follows from the multiplicative 
property of $\kap(q)$ that the estimate (\ref{8.5}) again holds. The conclusion of the lemma 
follows by substituting these estimates into (\ref{11.2}).
\end{proof}

\section{The minor arc contribution for systems of type B}
Owing to the presence of smooth Weyl sums within the generating function $\calF^*(\bfalp)$, our 
treatment of the minor arcs for systems of type B involves some pruning exercises. We begin by 
examining a set of minor arcs of large height.

\begin{lemma}\label{lemma9.1} One has $\displaystyle{\int_\grm |\calF^*(\bfalp)|\d\bfalp \ll 
P^{s-5-\del}}$.
\end{lemma}

\begin{proof} We divide into two cases according to whether $n=0$ or $n=3$. Suppose first that $n=0$. 
Then on recalling the definition of $\calF^*(\bfalp)$ and applying (\ref{6.1}), we deduce that for 
some integers $i$ and $j$, one has
\begin{equation}\label{9.1}
\int_\grm |\calF^*(\bfalp)|\d\bfalp \ll \int_\grm f_i^{s-m}\gtil_j^m\d\bfalp .
\end{equation}
In view of our hypotheses concerning $m$, by repeated application of (\ref{6.1}), as in the proof of 
\cite[Lemma 7.3]{Woo1991b}, one obtains the bound
\begin{equation}\label{9.2}
\int_\grm f_i^{s-m}\gtil_j^m\d\bfalp \ll \int_\grm (f_i^s+f_i^{s-5}\gtil_j^5)\d\bfalp .
\end{equation}
Thus it follows from (\ref{9.1}) that
\begin{equation}\label{9.3}
\int_\grm|\calF^*(\bfalp)|\d\bfalp \ll \Bigl( \sup_{\bfalp \in \grm}f_i\Bigr)^{s-32/3}I_1+
\Bigl( \sup_{\bfalp \in \grm}f_i\Bigr)^{s-91/9}I_2,
\end{equation}
where
$$I_1=\oint f_i^{32/3}\d\bfalp \quad \text{and}\quad I_2=\oint f_i^{46/9}\gtil^5\d\bfalp .$$

\par Lemma \ref{lemma5.1}(i) shows that $I_1\ll P^{17/3+\eps}$. Meanwhile, by applying H\"older's 
inequality in combination with Lemmata \ref{lemma5.1}(i) and \ref{lemma8.2}(i), one sees that
$$I_2\le \Bigl( \oint f_i^{32/3}\d\bfalp \Bigr)^{1/6}\Bigl( \oint f_i^4\gtil_j^6\d\bfalp \Bigr)^{5/6}
\ll (P^{17/3+\eps})^{1/6}(P^{21/4-2\del})^{5/6}.$$
Lemma \ref{lemma5.2} shows that $f_i=O(P^{3/4+\eps})$ for $\bfalp\in \grm$, and thus (\ref{9.3}) 
yields
$$\int_\grm|\calF^*(\bfalp)|\d\bfalp \ll P^\eps 
\left( (P^{3/4})^{s-32/3}P^{17/3}+(P^{3/4})^{s-91/9}P^{383/72}\right) \ll P^{s-5-\del}.$$
This completes the proof of the lemma when $n=0$.\par

Suppose next that $n=3$. Then on recalling the definition of $\calF^*(\bfalp)$ and applying 
(\ref{6.1}), we deduce that for some integers $i,j,k$, one has
\begin{equation}\label{12.1}
\int_\grm |\calF^*(\bfalp)|\d\bfalp \ll \int_\grm f_i^{s-m-3}\gtil_j^mh_k^3\d\bfalp .
\end{equation}
In view of our hypotheses concerning $m$, by repeated application of (\ref{6.1}), as in the proof of 
\cite[Lemma 7.3]{Woo1991b}, one obtains the bound
\begin{equation}\label{12.2}
\int_\grm f_i^{s-m-3}\gtil_j^mh_k^3\d\bfalp \ll \int_\grm 
(f_i^{s-4}\gtil_jh_k^3+f_i^{s-8}\gtil_j^5h_k^3)\d\bfalp .
\end{equation}
Thus it follows from (\ref{12.1}) that
\begin{equation}\label{12.3}
\int_\grm |\calF^*(\bfalp)|\d\bfalp \ll \Bigl( \sup_{\bfalp \in \grm}f_i\Bigr)^{s-32/3}I_1
+\Bigl( \sup_{\bfalp \in \grm}f_i\Bigr)^{s-91/9}I_2,
\end{equation}
where
$$I_1=\oint f_i^{20/3}\gtil_jh_k^3\d\bfalp \quad \text{and}\quad 
I_2=\oint f_i^{19/9}\gtil_j^5h_k^3\d\bfalp .$$

\par Applying H\"older's inequality with Lemmata \ref{lemma5.1}(ii) and \ref{lemma8.2}(ii), one 
obtains
$$I_1\le \biggl( \oint f_i^8h_k^2\d\bfalp \biggr)^{5/6} \biggl( \oint \gtil_j^6h_k^8\d\bfalp 
\biggr)^{1/6}\ll (P^{5+\eps})^{5/6}\left( P^4(P^{21/4-2\del})\right)^{1/6}.$$
Meanwhile, by applying H\"older's inequality now with Lemmata \ref{lemma5.1}(i) and \ref{lemma8.2}(i), 
(ii), we obtain the estimate
\begin{align*}
I_2&\le \biggl( \oint f_i^{32/3}\d\bfalp \biggr)^{1/6} \biggl( \oint f_i^4\gtil_j^6\d\bfalp 
\biggr)^{1/12}\biggl( \oint \gtil_j^6h_k^4\d\bfalp \biggr)^{3/4} \\
&\ll (P^{17/3+\eps})^{1/6}(P^{21/4-2\del})^{1/12}(P^{21/4-2\del})^{3/4}.
\end{align*}
Lemma \ref{lemma5.2} shows that $f_i=O(P^{3/4+\eps})$ for $\bfalp\in \grm$, and thus (\ref{12.3}) 
yields
$$\int_\grm |\calF^*(\bfalp)|\d\bfalp \ll (P^{3/4})^{s-32/3}P^{6-7/24}+(P^{3/4})^{s-91/9}P^{6-49/72}
\ll P^{s-5-\del}.$$
This completes the proof of the lemma when $n=3$.
\end{proof}

We next prune down to a narrow set of major arcs. Take $W=(\log \log P)^{100}$, and then put 
$\grQ=\grP(W)$ and $\grq=\grp(W)$. We seek to estimate the contribution of the set 
$\grM\setminus \grQ$ within the integral giving $R^*(P)$.

\begin{lemma}\label{lemma9.2} One has 
$\displaystyle{\int_{\grM\setminus \grQ}|\calF^*(\bfalp)|\d\bfalp \ll P^{s-5}(\log \log P)^{-1}}$.
\end{lemma}

\begin{proof} We again divide into two cases according to whether $n=0$ or $n=3$. Suppose first that 
$n=0$. The procedure leading to (\ref{9.2}) shows that for some $i$ and $j$, one has
\begin{equation}\label{9.4}
\int_{\grM\setminus \grQ}|\calF^*(\bfalp)|\d\bfalp \ll I_1+I_2,
\end{equation}
where
$$I_1=\int_{\grM\setminus \grQ}f_i^s\d\bfalp \quad \text{and}\quad I_2=\int_{\grM\setminus \grQ}
f_i^{s-5}\gtil_j^5\d\bfalp .$$
On recalling (\ref{7.1}), we see from Lemma \ref{lemma7.3} that whenever $\bfalp \in \grM$, one has 
$f_i(\bfalp)=f_i^*(\bfalp)+O(P^{1/2+\eps})$. Thus we deduce that
$$I_1\ll \Bigl( \sup_{\bfalp \in \grM\setminus \grQ}|f_i^*(\bfalp)|\Bigr)^{s-8}\oint 
|f_i^*|^8\d\bfalp +P^{s/2+\eps}.$$
Since $s\ge 11$, we infer from Lemma \ref{lemma8.3}(i), (ii) that
$$I_1\ll (PW^{-1/4})^{s-8}P^3+P^{s/2+\eps}\ll P^{s-5}W^{-3/4}\ll P^{s-5}(\log \log P)^{-1}.$$

\par In a similar manner, we obtain the bound $I_2\ll I_3+P^{(s-5)/2+\eps}I_4$, where
$$I_3=\int_{\grM\setminus \grQ}|f_i^*|^{s-5}\gtil_j^5\d\bfalp \quad \text{and}\quad 
I_4=\oint \gtil_j^5\d\bfalp .$$
H\"older's inequality leads via Lemma \ref{lemma8.3}(i), (ii), (iii) to the bound
\begin{align*}
I_3&\ll \bigl( PW^{-1/4}\bigr)^{s-1339/124}\biggl( \oint |f_i^*|^{29/4}\d\bfalp \biggr)^{11/31} 
\biggl( \oint |f_i^*|^5\gtil_j^{31/4}\d\bfalp \biggr)^{20/31},\\
&\ll (PW^{-1/4})^{s-1339/124}(P^{9/4})^{11/31}(P^{31/4})^{20/31}\ll P^{s-5}W^{-1/20}.
\end{align*}
Meanwhile, H\"older's inequality combines with Lemma \ref{lemma8.1} to give
$$I_4\le \biggl( \int_0^1 \gtil_j^6\d\alp_3\biggr)^{5/6}\ll (P^{13/4-2\del})^{5/6}\ll P^{11/4-\del }.$$
Thus we deduce that
$$I_2\ll P^{s-5}(\log \log P)^{-1}+P^{(2s+1)/4}\ll P^{s-5}(\log \log P)^{-1}.$$
When $n=0$, the lemma now follows by substituting these estimates into (\ref{9.4}).\par

Suppose next that $n=3$. Then the procedure leading to (\ref{12.2}) shows that for some $i,j,k$, one 
has
\begin{equation}\label{12.4}
\int_{\grM\setminus \grQ}|\calF^*(\bfalp)|\d\bfalp \ll I_1+I_2,
\end{equation}
where
$$I_1=\int_{\grM\setminus \grQ}f_i^{s-4}\gtil_jh_k^3\d\bfalp\quad \text{and}\quad 
I_2=\int_{\grM\setminus \grQ}f_i^{s-8}\gtil_j^5h_k^3\d\bfalp .$$

\par On recalling (\ref{7.1}), we see from Lemma \ref{lemma7.3} that whenever $\bfalp\in \grM$, one 
has $f_i(\bfalp)=f_i^*(\bfalp)+O(P^{1/2+\eps})$. Thus we deduce that $I_1\ll I_3+P^{(s-4)/2+\eps}I_4$, 
where
$$I_3=\oint |f_i^*|^{s-4}\gtil_jh_k^3\d\bfalp \quad \text{and}\quad I_4=\oint \gtil_jh_k^3\d\bfalp .$$
An application of H\"older's inequality yields
\begin{align*}
I_3\le &\, \Bigl( \sup_{\bfalp\in \grM\setminus \grQ}|f_i^*(\bfalp)|\Bigr)^{s-4640/527}
\biggl( \oint |f_i^*|^8\d\bfalp \biggr)^{5/17}\\
&\, \times \biggl( \oint \gtil_j^{31/4}h_k^{17/4}\d\bfalp \biggr)^{4/31}
\biggl( \oint |f_i^*h_k|^{17/4}\d\bfalp \biggr)^{304/527},
\end{align*}
and thus we deduce from Lemmata \ref{lemma8.2}(iii) and \ref{lemma8.3}(i), (ii), (iv) that
\begin{align*}
I_3\ll (PW^{-1/4})^{s-4640/527}(P^3)^{5/17}(P^7)^{4/31}(P^{7/2})^{304/527}\ll P^{s-5}W^{-1/2}.
\end{align*}
Meanwhile, Hua's lemma combines with H\"older's inequality to give
$$I_4\le \biggl( \int_0^1 \gtil_j^4\d\alp_3\biggr)^{1/4}\biggl( \int_0^1 h_k^4\d\alp_2\biggr)^{3/4}
\ll P^{2+\eps}.$$
Thus we conclude that
$$I_1\ll P^{s-5}(\log \log P)^{-1}+P^{s/2+\eps}\ll P^{s-5}(\log \log P)^{-1}.$$

We find in like manner that $I_2\ll I_5+P^{(s-8)/2+\eps}I_6$, where
$$I_5=\int_{\grM\setminus \grQ}|f_i^*|^{s-8}\gtil_j^5h_k^3\d\bfalp\quad \text{and}\quad 
I_6=\oint \gtil_j^5h_k^3\d\bfalp .$$
An application of H\"older's inequality yields
\begin{align*}
I_5\le &\,\Bigl( \sup_{\bfalp\in \grM\setminus \grQ}|f_i^*(\bfalp)|\Bigr)^{s-5592/527}
\biggl( \oint |f_i^*|^8\d\bfalp \biggr)^{5/17}\\
&\, \times \biggl( \oint \gtil_j^{31/4}h_k^{17/4}\d\bfalp \biggr)^{20/31}
\biggl( \oint |f_i^*h_k|^{17/4}\d\bfalp \biggr)^{32/527},
\end{align*}
and thus we deduce from Lemmata \ref{lemma8.2}(iii) and \ref{lemma8.3}(i), (ii), (iv) that
\begin{align*}
I_5\ll (PW^{-1/4})^{s-5592/527}(P^3)^{5/17}(P^7)^{20/31}(P^{7/2})^{32/527}\ll P^{s-5}W^{-1/11}.
\end{align*}
Since Lemma \ref{lemma8.1} and Hua's lemma combine with H\"older's inequality to give
$$I_6\le \biggl( \int_0^1\gtil_j^6\d\alp_3\biggr)^{5/6}\biggl( \int_0^1h_k^4\d\alp_2\biggr)^{3/4} 
\ll (P^{13/4-2\del})^{5/6}(P^{2+\eps})^{3/4},$$
we thus conclude that
$$I_2\ll P^{s-5}(\log \log P)^{-1}+P^{(2s+1)/4}\ll P^{s-5}(\log \log P)^{-1}.$$
When $n=3$, the lemma now follows by inserting these estimates into (\ref{12.4}).
\end{proof}

Since $\grq=\grm\cup (\grM\setminus \grQ)$, the estimates supplied by Lemmata \ref{lemma9.1} and 
\ref{lemma9.2} combine to give the following conclusion.

\begin{lemma}\label{lemma9.3} One has 
$\displaystyle{\int_\grq |\calF^*(\bfalp)|\d\bfalp \ll P^{s-5}(\log \log P)^{-1}}$.
\end{lemma}

\section{The major arc contribution for systems of type B}
The discussion of the major arcs $\grQ$ for systems of type B requires an asymptotic analysis of the 
generating function $\gtil_j(\bet)$.

\begin{lemma}\label{lemma10.1}
There is a positive number $c_\eta$ with the property that, whenever 
$\bfalp\in \grM(q,\bfr;R)\subseteq \grM(R)$, then
\begin{equation}\label{10.1}
\ftil_i(\bfalp)-c_\eta f_i^*(\bfalp)\ll P(\log P)^{-1}(q+P^2|q\alp_2-r_2|+P^3|q\alp_3-r_3|)
\end{equation}
and
\begin{equation}\label{10.2}
\gtil_j(\alp_3)-c_\eta g_j^*(\bfalp)\ll P(\log P)^{-1}(q+P^3|q\alp_3-r_3|).
\end{equation}
\end{lemma}

\begin{proof} On writing $\rho(t)$ for the Dickman function, one sees that
$$\int_{\xi_iP/2}^{2\xi_iP}\rho\left( \frac{\log \gam}{\log R}\right) 
e(a_i\bet_3\gam^3+b_i\bet_2\gam^2)\d\gam =\rho(1/\eta) v_{f,i}(\bfbet)+O(P/\log P).$$
Thus, by reference to \cite[Lemma 8.5]{Woo1991b}, we see that the relation (\ref{10.1}) follows with 
$c_\eta=\rho(1/\eta)>0$. The second conclusion (\ref{10.2}) follows in like manner.
\end{proof}

We now complete the analysis of the major arcs for systems of type B. From Lemmata \ref{lemma7.3} 
and \ref{lemma10.1}, we see that when $\bfalp\in \grQ(q,\bfr)\subseteq \grQ$, one has 
$$f_i(\bfalp)-f_i^*(\bfalp)\ll W,\quad \gtil_j(\alp_3)-c_\eta g_j^*(\bfalp)\ll PW^{-10},$$
$$h_k(\alp_2)-h_k^*(\bfalp)\ll W,$$
whence
$$\calF^*(\bfalp)-c_\eta^mT(q,\bfr)V(\bfalp -\bfr/q)\ll P^sW^{-10}.$$
The measure of the set of arcs $\grQ$ is $O(W^5P^{-5})$, and thus we conclude that
$$\int_\grQ \calF^*(\bfalp)\d\bfalp -c_\eta^m \grS(W)\grJ(W)\ll P^{s-5}W^{-5}.$$
We therefore deduce from Lemmata \ref{lemma7.4} and \ref{lemma7.5} that
$$\int_\grQ \calF^*(\bfalp)\d\bfalp =c_\eta^m\grS\calC P^{s-5}+O(P^{s-5}W^{-1/6}).$$
But $[0,1)^2$ is the union of $\grQ$ and $\grq$, and so we infer from Lemma \ref{lemma9.3} that
$$R^*(P)=\iint_{[0,1)^2}\calF^*(\bfalp)\d\bfalp =c_\eta^m\grS \calC P^{s-5}+
O(P^{s-5}(\log \log P)^{-1}).$$
Since $\calC\grS>0$, we deduce that $R^*(P)\gg P^{s-5}$. The conclusion of Theorem 
\ref{theorem1.2}, 
and hence also Theorem \ref{theorem1.1}, therefore follows for systems of type B. 

\section{Auxiliary estimates for systems of type C} Finally, we estimate $R^\dagger(P)$ when $n=3$ and 
$m=0$. We begin with some auxiliary estimates.

\begin{lemma}\label{lemma14.1} For all $i$ and $k$, one has
\begin{enumerate}
\item[(i)] ${\displaystyle{\oint \ftil_i^6h_k^4\d\bfalp \ll P^{21/4-2\del},}}$
\item[(ii)] ${\displaystyle{\oint \ftil_i^8h_k^{17/4}\d\bfalp \ll P^{29/4}.}}$
\end{enumerate}
\end{lemma}

\begin{proof} We begin with the estimate (i). By orthogonality, the mean value $I$ in question is 
bounded above by the number of integral solutions of the system
\begin{equation}\label{14.1}
\left. \begin{aligned}
a_i\sum_{u=1}^3(x_u^3-y_u^3)&=0\\
b_i\sum_{u=1}^3(x_u^2-y_u^2)&=d_k\sum_{u=4}^5(x_u^2-y_u^2)
\end{aligned}\right\} ,
\end{equation}
with $1\le \bfx,\bfy\le P$ and $x_u,y_u\in \calA(P,R)$ $(1\le u\le 3)$. Suppose that $\bfx,\bfy$ is a 
solution of (\ref{14.1}) counted by $I$. By applying Lemma \ref{lemma8.1} to the cubic equation in 
(\ref{14.1}), one sees that the number $I_1$ of possible choices for $x_u,y_u$ $(1\le u\le 3)$ 
satisfies $I_1=O(P^{13/4-3\del})$. Fix any one such choice. Then it follows from the quadratic 
equation in (\ref{14.1}) that $d_k(x_4^2+x_5^2-y_4^2-y_5^2)=N$. Hence, by the triangle inequality 
in combination with Hua's lemma, one sees that the number $I_2$ of choices for $x_4,x_5,y_5,y_6$ 
satisfies $I_2=O(P^{2+\eps})$. Thus $I\ll (P^{13/4-3\del})(P^{2+\eps})\ll P^{21/4-2\del}$, 
confirming the estimate (i).\par

We turn next to the estimate (ii). By orthogonality combined with \cite[Theorem 1]{Vau1986}, one 
obtains the bound
$$\int_0^1|\ftil_i(\alp_2,\alp_3)|^8\d\alp_3\le \int_0^1|g(\alp;P)|^8\d\alp \ll P^5.$$
We therefore deduce from an estimate of the type (\ref{11.1}) that
$$\int_0^1 |h_k(\alp_2)|^{17/4}\int_0^1|\ftil_i(\bfalp)|^8\d\alp_3\d\alp_2 
\ll P^5\int_0^1h_k^{17/4}\d\alp_2 \ll P^{29/4},$$
confirming the estimate (ii).
\end{proof}

\section{The minor arc estimate for systems of type C} The presence of the smooth Weyl sum 
$\ftil_l(\bfalp)$ in the generating function $\calF^\dagger(\bfalp)$ involves us again in some pruning 
exercises. We adopt the notation of \S9, and begin by examining a set of minor arcs of large height.

\begin{lemma}\label{lemma15.1} One has 
$\displaystyle{\int_\grm |\calF^\dagger(\bfalp)|\d\bfalp \ll P^{s-5-\del}}$.
\end{lemma}

\begin{proof} Recall that in systems of type $C$, one has $m=0$ and $n=3$. Thus, on recalling the 
definition of $\calF^\dagger(\bfalp)$ and applying (\ref{6.1}), we deduce that for some integers $i$ 
and $k$, one has
\begin{equation}\label{15.1}
\int_\grm |\calF^\dagger(\bfalp)|\d\bfalp \ll \int_\grm f_i^{s-4}\ftil_lh_k^3\d\bfalp .
\end{equation}
Then it follows from H\"older's inequality and the trivial estimate $h_k\le P$ that
$$\int_\grm |\calF^\dagger(\bfalp)|\d\bfalp \ll P^{2/3}\Bigl( \sup_{\bfalp\in \grm}f_i\Bigr)^{s-32/3}
\biggl( \oint f_i^8h_k^2\d\bfalp \biggr)^{5/6}\biggl( \oint \ftil_l^6h_k^4\d\bfalp \biggr)^{1/6}.$$
Lemma \ref{lemma5.2} shows that $f_i=O(P^{3/4+\eps})$ for $\bfalp\in \grm$, and hence by applying 
Lemmata \ref{lemma5.1}(ii) and \ref{lemma14.1}(i), one obtains
$$\int_\grm |\calF^\dagger (\bfalp)|\d\bfalp \ll P^{2/3+\eps}(P^{3/4})^{s-32/3}(P^{5+\eps})^{5/6}
(P^{21/4-2\del})^{1/6}\ll P^{s-5-\del}.$$
This completes the proof of the lemma.
\end{proof}

We next prune down to the narrow set of arcs $\grQ$.

\begin{lemma}\label{lemma15.2} One has 
$\displaystyle{\int_{\grM\setminus \grQ}|\calF^\dagger (\bfalp)|\d\bfalp \ll P^{s-5}(\log \log P)^{-1}}$.
\end{lemma}

\begin{proof} The procedure leading to (\ref{15.1}) shows that for some $i$ and $k$, one has
$$\int_{\grM\setminus \grQ}|\calF^\dagger(\bfalp)|\d\bfalp \ll 
\int_{\grM\setminus \grQ}f_i^{s-4}\ftil_lh_k^3\d\bfalp .$$
On recalling (\ref{7.1}), we see from Lemma \ref{lemma7.3} that whenever $\bfalp \in \grM$, one has 
$f_i(\bfalp)=f_i^*(\bfalp)+O(P^{1/2+\eps})$. Thus we deduce that
$$\int_{\grM\setminus \grQ}|\calF^\dagger(\bfalp)|\d\bfalp \ll I_1+P^{(s-4)/2+\eps}I_2,$$
where
$$I_1=\int_{\grM\setminus \grQ}|f_i^*|^{s-4}\ftil_lh_k^3\d\bfalp \quad \text{and}\quad 
I_2=\oint \ftil_lh_k^3\d\bfalp .$$
An application of H\"older's inequality yields
\begin{align*}
I_1\le &\, \Bigl( \sup_{\bfalp\in \grM\setminus \grQ}|f_i^*(\bfalp)|\Bigr)^{s-4799/544}
\biggl( \oint |f_i^*|^8\d\bfalp \biggr)^{5/17}\\
&\times \biggl( \oint |f_i^*h_k|^{17/4}\d\bfalp \biggr)^{79/136}
\biggl( \oint \ftil_l^8h_k^{17/4}\d\bfalp \biggr)^{1/8},
\end{align*}
and thus we deduce from Lemmata \ref{lemma8.3}(i), (ii), (iv) and \ref{lemma14.1}(ii) that
$$I_1\ll (PW^{-1/4})^{s-4799/544}(P^3)^{5/17}(P^{7/2})^{79/136}(P^{29/4})^{1/8}
\ll P^{s-5}W^{-1/2}.$$
A second application of H\"older's inequality leads via Hua's lemma to the bound
\begin{align*}
I_2\ll \biggl( \int_0^1 h_k^4\d\bfalp\biggr)^{3/4}\biggl( \oint \ftil_l^4\d\bfalp \biggr)^{1/4}
\ll P^{2+\eps}.
\end{align*}
Then we may conclude that
$$\int_{\grM\setminus \grQ}|\calF^\dagger(\bfalp)|\d\bfalp \ll P^{s-5}(\log \log P)^{-1}+P^{s/2+\eps}
\ll P^{s-5}(\log \log P)^{-1}.$$
This completes the proof of the lemma.
\end{proof}

Since $\grq=\grm\cup (\grM\setminus \grQ)$, the estimates supplied by Lemmata \ref{lemma15.1} and 
\ref{lemma15.2} combine to give the following conclusion.

\begin{lemma}\label{lemma15.3} One has 
$\displaystyle{\int_\grq |\calF^\dagger(\bfalp)|\d\bfalp \ll P^{s-5}(\log \log P)^{-1}}$.
\end{lemma}

\section{The major arc contribution for systems of type C} The analysis of the major arcs may be 
completed for systems of type C by adapting the corresponding discussion of \S10. First, by applying 
Lemmata \ref{lemma7.3} and \ref{lemma10.1}, we see that when 
$\bfalp\in \grQ(q,\bfr)\subseteq \grQ$, then
$$\calF^\dagger(\bfalp)-c_\eta T(q,\bfr)V(\bfalp-\bfr/q)\ll P^sW^{-10},$$
and hence one deduces that
$$\int_\grQ \calF^\dagger(\bfalp)\d\bfalp =c_\eta \grS \calC P^{s-5}+O(P^{s-5}W^{-1/6}).$$
But $[0,1)^2$ is the union of $\grQ$ and $\grq$, and so we infer from Lemma \ref{lemma15.3} that 
$$R^\dagger(P)=\iint_{[0,1)^2}\calF^\dagger(\bfalp)\d\bfalp 
=c_\eta \grS\calC P^{s-5}+O(P^{s-5}(\log \log P)^{-1}).$$
Since $\calC\grS>0$, we deduce that $R^\dagger(P)\gg P^{s-5}$. The conclusion of Theorem 
\ref{theorem1.2}, and hence also Theorem \ref{theorem1.1}, therefore follows for systems of the 
final type C.

\section{Appendix: a transference principle} We take the opportunity here of establishing a standard 
transference principle. The version of this principle described in Exercise 2 of \cite[\S2.8]{Vau1997} 
restricts attention to the situation relevant to Weyl's inequality, and contains an additional condition 
on the relevant Diophantine approximation, and so it seems worthwhile to provide in the literature a 
complete account for future reference.   

\begin{lemma}\label{lemma16.1} Let $\tet, X,Y,Z$ be positive real numbers. Suppose that 
$\Psi:\dbR\rightarrow \dbC$ satisfies the property that whenever $a\in \dbZ$ and $q\in \dbN$ satisfy 
$(a,q)=1$ and $|\alp-a/q|\le q^{-2}$, then
\begin{equation}\label{16.1}
\Psi(\alp)\ll X(q^{-1}+Y^{-1}+qZ^{-1})^\tet.
\end{equation}
Then, whenever $b\in \dbZ$ and $r\in \dbN$ satisfy $(b,r)=1$, one has
\begin{equation}\label{16.2}
\Psi(\alp)\ll X(\lam^{-1}+Y^{-1}+\lam Z^{-1})^\tet ,
\end{equation}
where $\lam=r+Z|r\alp-b|$.
\end{lemma}

\begin{proof} Suppose that $b\in \dbZ$ and $r\in \dbN$ satisfy $(b,r)=1$. By Dirichlet's theorem on 
Diophantine approximation, there exist $a\in \dbZ$ and $q\in \dbN$ with $1\le q\le 2r$ and 
$|q\alp-a|\le (2r)^{-1}$. Suppose in the first instance that $a/q\ne b/r$. Then
$$(qr)^{-1}\le \left| \frac{a}{q}-\frac{b}{r}\right| \le \left| \alp-\frac{b}{r}\right| 
+\left| \alp-\frac{a}{q}\right| \le \left| \alp-\frac{b}{r}\right| +(2qr)^{-1}.$$
Hence $q^{-1}\le 2|r\alp-b|$, and so we deduce from (\ref{16.1}) that
$$\Psi(\alp)\ll X\left( |r\alp-b|+Y^{-1}+rZ^{-1}\right)^\tet \ll X(Y^{-1}+\lam Z^{-1})^\tet ,$$
confirming the estimate (\ref{16.2}).\par

If, on the other hand, one has $a/q=b/r$, then since $(a,q)=(b,r)=1$, one has $q=r$ and $a=b$, and 
hence $|r\alp-b|\le (2r)^{-1}$. If $\alp=b/r$, then $\lam=r$, and the desired conclusion (\ref{16.2}) 
is immediate from (\ref{16.1}). When $\alp\ne b/r$, meanwhile, one has $0<|\alp-b/r|\le r^{-2}$. In 
this situation, by Dirichlet's theorem on Diophantine approximation, there exist $a\in \dbZ$ and 
$q\in \dbN$ with $1\le q\le 2|r\alp-b|^{-1}$ and $|q\alp-a|\le \tfrac{1}{2}|r\alp-b|$. If one were to 
have $a/q=b/r$, then since $(a,q)=(b,r)=1$, one finds that $q=r$ and $a=b$, and hence 
$0<|r\alp-b|\le \tfrac{1}{2}|r\alp-b|$, giving a contradiction. Thus $a/q\ne b/r$, and one has 
\begin{align*}
(qr)^{-1}&\le \left| \frac{a}{q}-\frac{b}{r}\right|\le \left|\alp-\frac{b}{r}\right|+
\left|\alp-\frac{a}{q}\right| \\
&\le \left|\alp-\frac{b}{r}\right| +(2q)^{-1}|r\alp-b|\le \left|\alp-\frac{b}{r}\right|+(2qr)^{-1}.
\end{align*}
We therefore see that $q^{-1}\le 2|r\alp-b|$, and we deduce from (\ref{16.1}) that
$$\Psi(\alp)\ll X\left( (Z|r\alp-b|)^{-1}+Y^{-1}+|r\alp-b|\right)^\tet .$$
Alternatively, since $|\alp-b/r|\le r^{-2}$, one may apply (\ref{16.1}) to give
$$\Psi(\alp)\ll X(r^{-1}+Y^{-1}+rZ^{-1})^\tet .$$
Thus, in any case, one obtains the bound $\Psi(\alp)\ll X(\lam^{-1}+Y^{-1}+\lam Z^{-1})^\tet$, where 
$\lam=r+Z|r\alp-b|$. This completes the proof of the lemma.
\end{proof}

\bibliographystyle{amsbracket}
\providecommand{\bysame}{\leavevmode\hbox to3em{\hrulefill}\thinspace}

\end{document}